\documentclass[11pt,leqno]{article}
\usepackage[english]{babel}
\usepackage{tikz}
\usepackage{enumerate}
\usepackage{graphicx}
\usepackage{verbatim}
\usepackage{verbatim}
\usepackage{etoolbox}
\usepackage{tkz-euclide}
\usepackage{geometry}
\usepackage{mathtools}
\usepackage{color}
\usepackage[title]{appendix}
\usepackage[all]{xy}
\usepackage{tikz-cd}
\usepackage{blindtext}
\usepackage[T1]{fontenc}
\usepackage{amsthm}
\usepackage{amsfonts}
\usepackage{txfonts}
\usepackage{palatino,amssymb,epsfig}
\usepackage{latexsym,epsf,epic,amscd}
\usepackage{mathrsfs}
\usepackage{graphicx}
\usepackage{caption}
\usepackage{subcaption}
\usepackage{appendix}
\usepackage{amsmath}
\usepackage{bookmark} 
\usepackage[nottoc]{tocbibind}
\hypersetup{
colorlinks=true,
linkcolor=black,
citecolor=black,
anchorcolor=black,
urlcolor=black}
\allowdisplaybreaks[4]
\numberwithin{equation}{section}

\DeclareMathOperator{\arccot}{arccot}
\DeclareMathOperator{\arccosh}{arccosh}

\newcommand{\pp}[2]{\frac{\partial#1}{\partial#2}}

    \newcommand{\Addresses}{{
  \bigskip
  \footnotesize
  \noindent Guangming Hu, \href{18810692738@163.com}{18810692738@163.com}
\newline\textit{College of Science, Nanjing University of Posts and Telecommunications,
  Nanjing, 210003, P.R. China.}\par\nopagebreak
  \medskip
 \noindent Ziping Lei, \href{zplei@ruc.edu.cn}{zplei@ruc.edu.cn}
  \newline\textit{ School of Mathematics, Renmin University of China, Beijing, 100872, P.R. China.} \par\nopagebreak
    \medskip
  \noindent Yu Sun, \href{yusun15185105160@163.com}{yusun15185105160@163.com}
  \newline\textit{School of mathematics and physics, Nanjing institute of technology, Nanjing, 211100, P.R. China.} \par\nopagebreak
    \medskip
  \noindent Puchun Zhou, \href{pczhou22@m.fudan.edu.cn}{pczhou22@m.fudan.edu.cn}
  \newline\textit {School of Mathematical Sciences, Fudan University, Shanghai, 200433, P.R. China} \par\nopagebreak
}}

\title{The Convergence of Prescribed Combinatorial Ricci Flows for  Total Geodesic Curvatures in Spherical Background Geometry}
\author{ Guangming Hu, Ziping Lei,Yu Sun and Puchun Zhou}

\date{}
\newtheorem{theorem}{Theorem}[section]
\newtheorem{lemma}[theorem]{Lemma}
\newtheorem{proposition}[theorem]{Proposition}
\newtheorem{corollary}[theorem]{Corollary}
\theoremstyle{definition}
\newtheorem{definition}[theorem]{Definition}
\newtheorem{remark}[theorem]{Remark}

\begin{document}
\maketitle

\begin{abstract}
In this paper, we study the existence and rigidity of (degenerated) circle pattern metric with prescribed total geodesic curvatures in spherical background geometry. To find the (degenerated) circle pattern metric with prescribed total geodesic curvatures, we define some prescribed combinatorial Ricci flows and study the convergence of flows for (degenerated) circle pattern metrics. We solve the prescribed total geodesic curvature problem and provide two methods to find the degenerated circle pattern metric with prescribed total geodesic curvatures. As far as we know, this is the first degenerated result for total geodesic curvatures in spherical background geometry.      

\end{abstract}

\section{Introduction}\label{sec1}
\subsection{Background}

In differential geometry, Hamilton \cite{Hamilton} introduced the Ricci flow defined by some equations $\frac{dg_{ij}}{dt}=-2Kg_{ij}$, where $g_{ij}$ is the Riemannian metric and $K$ is the Gaussian curvature. This is an important tool to solve the Poincaré conjecture \cite{Perelman1, Perelman2, Perelman3} and Thurston’s
geometrization conjecture \cite{cao}. Moreover, the Ricci flow is used to prove the uniformization theorem \cite{Tian}. 

In discrete geometry, Chow and Luo \cite{chow} constructed the combinatorial Ricci flow which is an analogy of the Ricci flow in smooth case. Given a closed surface $S$ with a triangulation $\mathscr{T}=(V,E,F)$ and a weight $\Phi: E\to [0,\frac{\pi}{2}]$, then we can define a map $\mathbf{r}=(r_1,\cdots,r_{|V|}): V\to (0,+\infty)$, the $\mathbf{r}$ is called a metric on the weighted triangulation $(\mathscr{T},\Phi)$. Using the metric $\mathbf{r}$, for any edge $e_{ij}\in E$ joining $v_i$ and $v_j$, we can define the length $l_{ij}$ of $e_{ij}$ in spherical (Euclidean, hyperbolic) background geometry, i.e. 
$$
l_{ij}=\arccos(\cos r_i\cos r_j-\sin r_i \sin r_j\cos(\Phi(e_{ij})))~(\mathbb{S}^2),
$$ 
$$
l_{ij}=\sqrt{r_i^2+r_j^2+2r_i r_j\cos(\Phi(e_{ij}))}~(\mathbb{R}^2)
$$
and
$$
l_{ij}=\arccosh(\cosh r_i\cosh r_j+\sinh r_i\sinh r_j\cos(\Phi(e_{ij})))~(\mathbb{H}^2).
$$
Then there exists a local spherical (Euclidean, hyperbolic) metric on each triangle $\triangle_{ijk}\in F$ whose vertices are $v_i$, $v_j$ and $v_k$. Gluing these spherical (Euclidean, hyperbolic) triangles along their sides together, then we obtain a spherical (Euclidean, hyperbolic) conical metric with singularities at the vertices on the closed surface $S$. Denote $\vartheta_{i}^{jk}$ as the angle at the vertex $v_i$ in the spherical (Euclidean, hyperbolic) triangle $\triangle_{ijk}\in F$. The combinatorial Gaussian curvature $K_i$ at the vertex $v_i$ is defined as 
$$
K_i=2\pi-\sum_{\triangle_{ijk}\in F}\vartheta_{i}^{jk},
$$
where the sum is taken over all triangles with the vertex $v_i$ in $F$. Then Chow and Luo \cite{chow} defined the combinatorial Ricci flow in spherical (Euclidean, hyperbolic) background geometry, i.e.  
\begin{equation}
\frac{dr_i}{dt}=-K_i\sin r_i~(\mathbb{S}^2),
\end{equation}
\begin{equation}\label{ef}
\frac{dr_i}{dt}=-K_ir_i~(\mathbb{R}^2)
\end{equation}
and 
\begin{equation}\label{hf}
\frac{dr_i}{dt}=-K_i\sinh r_i~(\mathbb{H}^2).
\end{equation}
Chow and Luo \cite{chow} studied the combinatorial Ricci flow and obtained the long time existence of the solution to the combinatorial Ricci flow. Moreover, the solution to the combinatorial Ricci flow will converge exponentially fast to the metric with constant curvature in some cases. They obtained the following two theorems:
\begin{theorem}
Suppose $(\mathscr{T}, \Phi)$ is a weighted triangulation of a closed connected surface $S$. Given any initial metric based on the weighted triangulation, the solution to the combinatorial Ricci flow (\ref{ef}) in the Euclidean background geometry with the given initial value exists for all time and converges if and only if for any proper subset $I \subset V$,
\begin{equation}\label{con}
2 \pi |I| \chi(S) /|V|>-\sum_{(e, v) \in Lk(I)}(\pi-\Phi(e))+2\pi\chi(F_I),
\end{equation}
where $F_I$ is the subcomplex of $\mathscr{T}$ consisting of all simplex whose vertices are contained in $I$, $ (e, v)\in \operatorname{Lk}(I)$ is a triangle in $F$ with a vertex $v$ in $(e, v)$ belongs to $I$ and neither of the endpoints of the edge $e$ in $(e, v)$ opposite $v$ belongs to $I$ and $\operatorname{Lk}(I)$ is the set of all such triangles. 

Furthermore, if the solution converges, then it converges exponentially fast to the metric with constant curvature.
\end{theorem}

\begin{theorem}
Suppose $(\mathscr{T}, \Phi)$ is a weighted triangulation of a closed connected surface $S$ of negative Euler characteristic. Given any initial metric, the solution to (\ref{hf}) in the hyperbolic background geometry with the given initial value exists for all time and converges if and only if the following two conditions hold.
\begin{enumerate}
\item\label{t1} For any three edges $e_1, e_2, e_3$ forming a null homotopic loop in $S$, if $\sum_{i=1}^3 \Phi\left(e_i\right) \geq \pi$, then these three edges form the boundary of a triangle in $F$.
\item\label{t2} For any four edges $e_1, e_2, e_3, e_4$ forming a null homotopic loop in $S$, if $\sum_{i=1}^4 \Phi\left(e_i\right) \geq$ $2 \pi$, then these $e_i$ 's form the boundary of the union of two adjacent triangles in $F$.
\end{enumerate}

Furthermore, if it converges, then it converges exponentially fast to a hyperbolic metric on $S$ so that all vertex angles are $2 \pi$.
\end{theorem}

Besides, Chow and Luo \cite{chow} posed some questions and one of them is whether the limit $\lim_{t\to +\infty}r_i(t)$ of the solution $r_i(t)$ to combinatorial Ricci flow always exists in the extended set $[0,+\infty]$ when the combinatorial conditions (\ref{con}), \ref{t1} and \ref{t2} are not valid.  

Takatsu \cite{takasu} studied the question and solved it, based on an infinitesimal description of degenerated circle pattern metrics. She obtained the following theorem:
\begin{theorem}
Let $(\mathscr{T}, \Phi)$ be a weighted triangulation of a surface $S$ of nonpositive Euler characteristic such that $\phi(I) \leq 0$ holds for any subset $I \subset V$ and
$Z_T:=\{z \in V \mid$ there exists a proper subset $Z$ of $V$ such that $z \in Z$ and $\phi(Z)=0\}$ is nonempty, where 
$$
\phi(I):=-\sum_{(e, v) \in Lk(I)}(\pi-\Phi(e))+2 \pi \chi\left(F_I\right).
$$ 
Then for any metric $\mathbf{r}$ on $(\mathscr{T}, \Phi)$, the solution $\{\mathbf{r}(t)\}_{t \geq 0}$ to combinatorial Ricci flow with initial data $\mathbf{r}$ does not converge on $\mathbb{R}_{>0}^{|V|}$ at infinity. However,
$$
\lim _{t \rightarrow +\infty} K_i(\mathbf{r}(t))=0
$$
holds for any vertex $v_i$.

On the one hand, for $\chi(S)=0$, the solution $\{\mathbf{r}(t)\}_{t \geq 0}$ to combinatorial Ricci flow does not converge on $\mathbb{R}_{\geq 0}^{|V|}$ at infinity. However, if we fix an arbitrary $v \in V \backslash Z_T$, then the limit
$$
\rho_u:=\lim _{t \rightarrow +\infty} \frac{r_u(t)}{r_v(t)}
$$
exists for any $u \in V$, where $Z_T=\left\{z \in V \mid \rho_z=0\right\}$ holds and $\left(\rho_u\right)_{u \in V \backslash Z_T}$ is a unique circle pattern metric with normalization $\rho_v=1$ on a certain weighted triangulation with vertices $V \backslash Z_T$.

On the other hand, for $\chi(S)<0$, the solution $\{\mathbf{r}(t)\}_{t \geq 0}$ to combinatorial Ricci flow converges on $\mathbb{R}_{\geq 0}^{|V|}$ at infinity, where we have $Z_T=\left\{z \in V \mid \lim _{t \rightarrow \infty} r_z(t)=0\right\}$ holds and the limit of $\left(r_v(t)\right)_{v \in V \backslash Z_T}$ at infinity is a unique circle pattern metric on a certain weighted triangulation with vertices $V \backslash Z_T$.
\end{theorem}

Recently, Nie \cite{nie} defined a new combinatorial scalar curvature, i.e. the total geodesic curvature and constructed a convex functional with total geodesic curvature. He obtained the rigidity of circle patterns in spherical background geometry. Followed his work, the last author of this paper and his collaborators \cite{GHZ} defined a combinatorial curvature flow which is an analogy of the prescribed combinatorial Ricci flow. They obtained an algorithm to find the desired ideal circle pattern.

Motivated by \cite{chow, takasu}, we obtain some results for total geodesic curvature in spherical background geometry. In this paper, we will define the prescribed combinatorial Ricci flow for total geodesic curvature in spherical background geometry and study the convergence of the solution to prescribed combinatorial Ricci flow. 

\subsection{Main results}

\subsubsection{Spherical conical metrics on surfaces}

Given a closed topological surface $S$ and a cellular decomposition $\Sigma=(V,E,F)$ of $S$. Denote $V, E$ and $F$ as set of $0$-cells, $1$-cells and $2$-cells, respectively. By $|E|$ we denote the number of 1-cells and by $|F|$ we denote the number of 2-cells. We can define a function $\mathbf{r}$ on $V$ and a function $\Phi$ on $E$, i.e. $\mathbf{r}=(r_1,\cdots,r_{|F|}): V\to (0,\frac{\pi}{2}]$ and $\Phi: E\to (0,\frac{\pi}{2})$.  Besides, the function $\Phi$ is called a weight.   

   For any 2-cell $f\in F$, we choose a auxiliary point $p_f\in f$ and add an edge between each vertex on $\partial f$ and $p$. Then we obtain a triangulation of $S$ as shown in Figure \ref{fig2} .  
   
   \begin{figure}[htbp]
\centering
\includegraphics[scale=1.0]{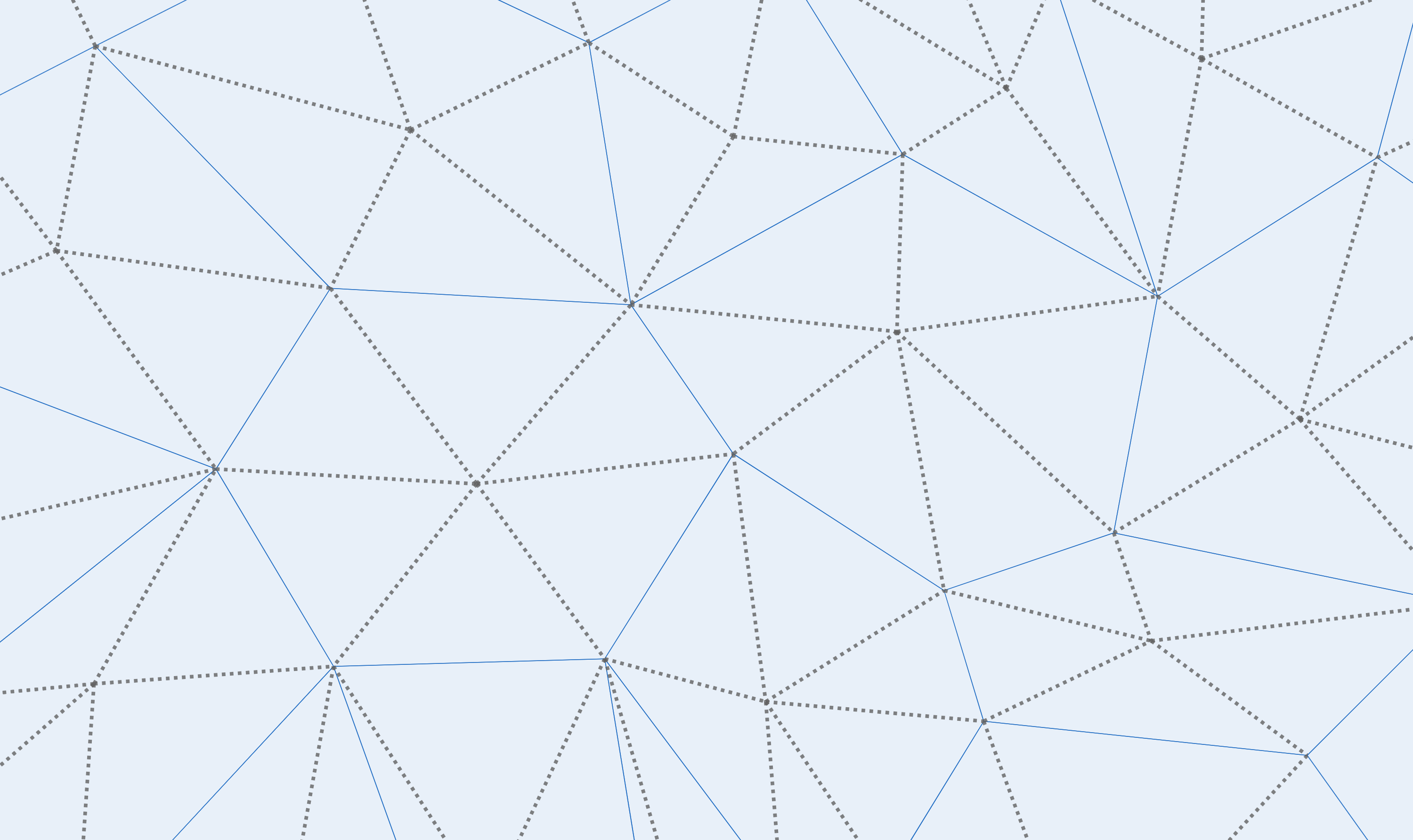}
\captionof{figure}{\small The cellular decomposition and triangulation of $S$ }
\label{fig2}
\end{figure} 
   
   For any 1-cell $e\in E$, there exists a quadrilateral $Q(e)$, where the vertices of $Q(e)$ are the end points of $e$ and auxiliary points of the 2-cells on the two sides of $e$. Denote the two auxiliary points as $p, {p}'$ and $f(p), f({p}')\in F$ denotes the 2-cells containing the points $p, {p}'$, respectively.        
   
   Given the weight $\Phi\in (0,\frac{\pi}{2})^{|E|}$, for any radii $\mathbf{r}\in (0,\frac{\pi}{2}]^{|F|}$, by Lemma \ref{lemma1}, there exists a local spherical metric on the quadrilateral $Q(e)$ such that $Q(e)$ is the corresponding spherical quadrilateral of two spherical disks, where the radii of the disks are $r_{f(p)}$ and $r_{f({p}')}$ and their intersection angle is $\Phi(e)$ as shown in Figure \ref{fig3}.

\begin{figure}[htbp]
\centering
\includegraphics[scale=1.6]{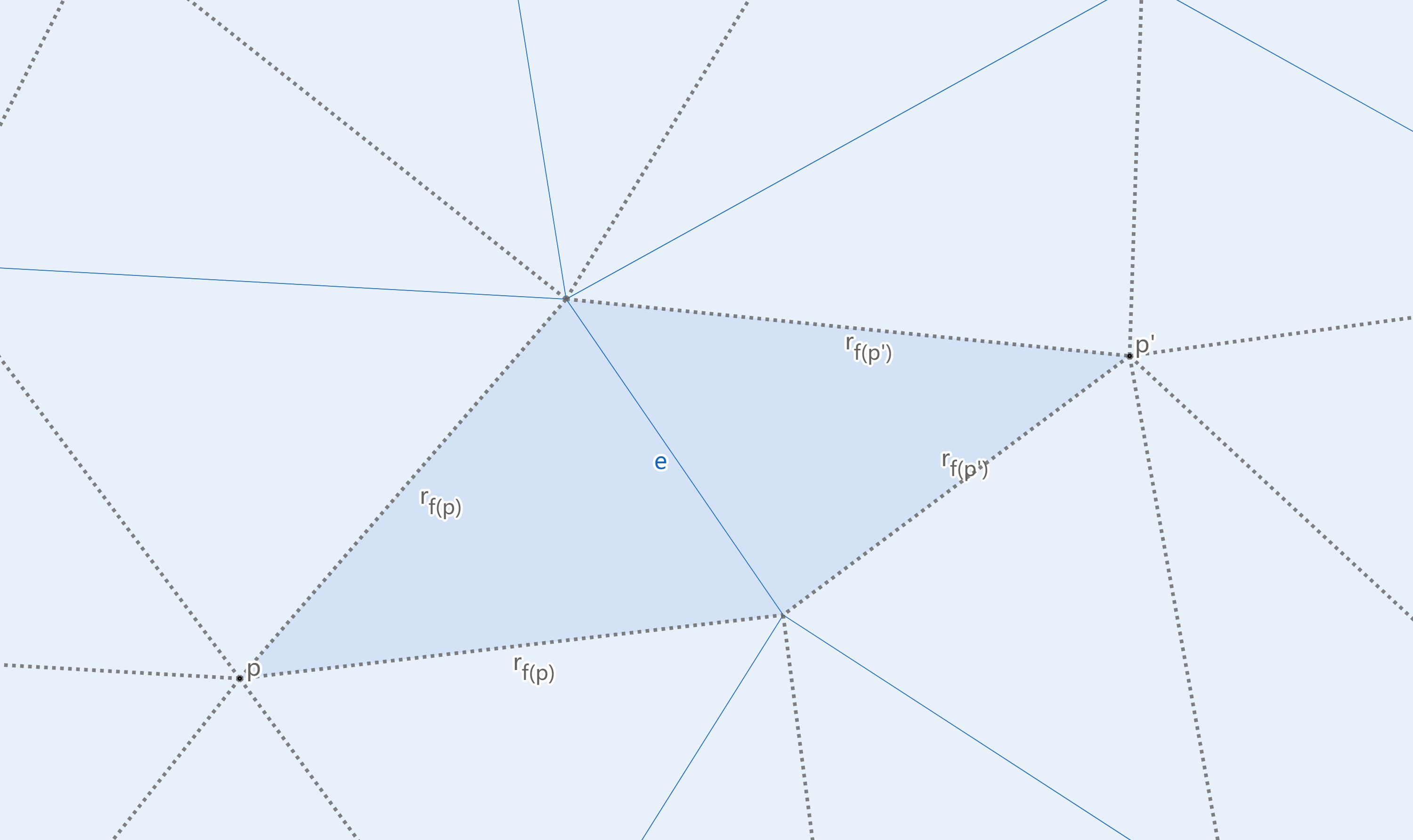}
\captionof{figure}{\small The spherical quadrilateral $Q(e)$ }
\label{fig3}
\end{figure} 
   Gluing these spherical quadrilaterals along their sides together, then we obtain a spherical conical metric on the closed topological surface $S$. Denote the spherical conical metric as $\mu$, we obtain a spherical conical surface $(S,\mu)$. 

   By definition, these spherical disks form a circle pattern $\mathcal{P}$ on $(S,\mu)$ which realizes the weight $\Phi$. By $(\mu, \mathcal{P})$ we denote the spherical conical metric $\mu$ with circle pattern $\mathcal{P}$. Besides, if the radii $\mathbf{r}\in (0,\frac{\pi}{2})^{|F|}$, then the spherical conical metric $\mu$ is called a $\textbf{circle~pattern~metric}$. If there exists some radii equal to $\frac{\pi}{2}$, then the spherical conical metric $\mu$ is called a $\textbf{degenerated~circle~pattern~metric}$.    

\begin{remark}
Given a closed topological surface $S$ with a cellular decomposition $\Sigma$, by the above argument, for the weight $\Phi\in (0,\frac{\pi}{2})^{|E|}$ and the radii $\mathbf{r}\in (0,\frac{\pi}{2}]^{|E|}$, there exists the corresponding circle pattern $\mathcal{P}$ which realizes the weight $\Phi$ and the radii $\mathbf{r}$. 
\end{remark}

For a disk $D\subset \mathbb{S}^2$, suppose that the radius of $D$ is $r\in (0, \frac{\pi}{2}] $. Then the $\partial D$ has constant geodesic curvature $k(\partial D)=\cot r$. Hence we can define a map $\iota$, i.e.
 $$
\begin{array}{cccc}
\iota: (0,\frac{\pi}{2}]^{|F|} &\longrightarrow &\mathbb{R}^{|F|}_{\ge 0}\\
\mathbf{r}=(r_1,\cdots,r_{|F|})^T&\longmapsto & k=(\cot r_1, \cdots, \cot r_{|F|})^T\\
\end{array}
$$
The map $\iota$ is bijective. 

Besides, for simplicity, in this paper, if the geodesic curvatures $k\in\mathbb{R}_{\ge 0}^{|F|}$, then $k$ is also called a $\textbf{spherical~conical~metric}$. If the geodesic curvatures $k\in\mathbb{R}_{>0}^{|F|}$, then $k$ is also called a $\textbf{circle~pattern~metric}$. If there exists some geodesic curvatures equal to $0$, then $k$ is also called a $\textbf{degenerated~circle~pattern~metric}$.

For any prescribed total geodesic curvature $\hat{L}$, we want to know whether there exists some (degenerated) circle pattern metrics with total geodesic curvature $\hat{L}$. In other words, we consider the following problem:

\noindent\textbf{Prescribed Total Geodesic Curvature Problem:} Is there a (degenerated) circle pattern metric with the prescribed total geodesic curvature $\hat{L}$? If it exists, how to find it? 

Our first results solve the first part of prescribed total geodesic curvature problem:    

\begin{theorem}\label{thm10}
Given a closed topological surface $S$ with a cellular decomposition $\Sigma=(V,E,F)$ and the weight $\Phi\in (0,\frac{\pi}{2})^{|E|}$. For the prescribed total geodesic curvature $\hat{L}=(\hat{L}_1,\cdots,\hat{L}_{|F|})^T$ on the face set $F$, there exists a circle pattern metric on $S$ with total geodesic curvature $\hat{L}$ if and only if the prescribed total geodesic curvature $\hat{L}\in\mathcal{L}_1$, where 
$$
\mathcal{L}_1=\{(L_1,\cdots,L_{|F|})^T\in \mathbb{R}^{|F|}_{>0}~|~\sum_{f\in F'}L_f<2\sum_{e\in E_{F'}}\Phi(e), \forall F'\subset F\}.
$$ 
Moreover, the circle pattern metric is unique up to isometry if it exists.
\end{theorem}

\begin{remark}
    Theorem \ref{thm10} is first proved in \cite{nie} and we provide a concrete proof in this paper.
\end{remark} 
 
\begin{theorem}\label{thm11}
Given a closed topological surface $S$ with a cellular decomposition $\Sigma=(V,E,F)$ and the weight $\Phi\in (0,\frac{\pi}{2})^{|E|}$. For the prescribed total geodesic curvature $\hat{L}=(\hat{L}_1,\cdots,\hat{L}_{|F|})^T$ on the face set $F$, there exists a degenerated circle pattern metric on $S$ with total geodesic curvature $\hat{L}$ if and only if the prescribed total geodesic curvature $\hat{L}\in\bar{\partial}\mathcal{L}$, where
$$
\bar{\partial}\mathcal{L}=\bigcup_{1\le m\le |F|-1,1\le i_1<\cdots<i_m\le |F|}\mathcal{L}_{i_1\cdots i_m}\cup\{0\},
$$
$$
\mathcal{L}_{i_1\cdots i_m}=\left\{\begin{array}{l|l}
(0,\cdots,L_{i_1},\cdots,0,\cdots,L_{i_m},0,\cdots,0)^T\in\mathbb{R}^{|F|} & \begin{array}{l}
\sum_{f\in F_{i_1\cdots i_m}'}L_f<2\sum_{e\in E_{F_{i_1\cdots i_m}'}}\Phi(e),\\
\forall F_{i_1\cdots i_m}'\subset F_{i_1\cdots i_m};L_{i_j}>0, 1\le j\le m
\end{array}
\end{array}\right\}.
$$
Moreover, the degenerated circle pattern metric is unique up to isometry if it exists.
\end{theorem}

\begin{theorem}\label{thm12}
Given a closed topological surface $S$ with a cellular decomposition $\Sigma=(V,E,F)$ and the weight $\Phi\in (0,\frac{\pi}{2})^{|E|}$. For the prescribed total geodesic curvature $\hat{L}=(\hat{L}_1,\cdots,\hat{L}_{|F|})^T$ on the face set $F$, there exists a spherical conical metric on $S$ with total geodesic curvature $\hat{L}$ if and only if the prescribed total geodesic curvature $\hat{L}\in\mathcal{L}$, where $\mathcal{L}=\mathcal{L}_1\cup\bar{\partial}\mathcal{L}$. Moreover, the spherical conical metric is unique up to isometry if it exists.
\end{theorem}

\subsubsection{Prescribed combinatorial Ricci flows for total geodesic curvatures}

Given $\hat{L}=(\hat{L}_1,\cdots,\hat{L}_{|F|})^T\in\mathcal{L}_1$, then we can define the prescribed combinatorial Ricci flows as follows, i.e.
\begin{equation}\label{s1}
\frac{dk_i}{dt}=-(L_i-\hat{L}_i)k_i,~~\forall i\in F.
\end{equation}

For any $k\in\mathbb{R}_{i_1\cdots i_m}^{|F|}$, where $1\le m\le |F|-1$, $1\le i_1<\cdots<i_m\le |F|$ and 
$$
\mathbb{R}^{|F|}_{i_1\cdots i_m}=\{(0,\cdots,k_{i_1},0,\cdots,k_{i_m},\cdots,0)\in\mathbb{R}^{|F|}~|~k_{i_1},\cdots,k_{i_m}>0\},
$$
then $\tilde{k}=(k_{i_1},\cdots,k_{i_m})\in\mathbb{R}_{>0}^m$. Given $\hat{L}\in\mathcal{L}_{i_1\cdots i_m}$, we construct the prescribed combinatorial Ricci flows as follows, i.e.
\begin{equation}\label{s3}
\frac{d\tilde{k}_i}{dt}=-(L_i(\tilde{k})-\hat{L}_i)\tilde{k}_i,~~\forall i\in F_{i_1\cdots i_m},
\end{equation}  
where $F_{i_1\cdots i_m}=\{i_1,\cdots,i_m\}$.

Given the prescribed total geodesic curvature $\hat{L}\in\mathcal{L}_{i_1\cdots i_m}$, we can define the prescribed combinatorial Ricci flows as follows, i.e. 
\begin{equation}\label{s2}
\frac{dk_i}{dt}=-(L_i-\hat{L}_i)k_i,~~\forall i\in F_{i_1\cdots i_m},~
\frac{dk_i}{dt}=-L_ik_i,~~\forall i\in F\setminus F_{i_1\cdots i_m}.
\end{equation}

The following results solve the second part of prescribed total geodesic curvature problem:

\begin{theorem}\label{s4}
Given a closed topological surface $S$ with a cellular decomposition $\Sigma=(V,E,F)$, the weight $\Phi\in (0,\frac{\pi}{2})^{|E|}$ and the prescribed total geodesic curvature $\hat{L}=(\hat{L}_1,\cdots,\hat{L}_{|F|})^T\in\mathcal{L}_1$ on the face set $F$. For any initial geodesic curvature $k(0)\in\mathbb{R}_{>0}^{|F|}$, the solution of the prescribed combinatorial Ricci flows (\ref{s1}) converges to the unique circle pattern metric with the total geodesic curvature $\hat{L}$ up to isometry. 
\end{theorem}    

\begin{remark}

 Theorem \ref{s4} is first proved in \cite{GHZ} and we provide  a new proof in this paper.
\end{remark}

\begin{theorem}\label{s6}
Given a closed topological surface $S$ with a cellular decomposition $\Sigma=(V,E,F)$, the weight $\Phi\in (0,\frac{\pi}{2})^{|E|}$ and the prescribed total geodesic curvature $\hat{L}\in\mathcal{L}_{i_1\cdots i_m}(1\le m\le |F|-1$, $1\le i_1<\cdots<i_m\le |F|)$ on the face set $F$. For any initial geodesic curvature $\tilde{k}(0)\in\mathbb{R}_{>0}^m$ on the face set $F_{i_1\cdots i_m}$ and 0 on the face set $F\setminus F_{i_1\cdots i_m}$, the solution of the prescribed combinatorial Ricci flows (\ref{s3}) converges to the unique degenerated circle pattern metric with the total geodesic curvature $\tilde{L}=(\hat{L}_{i_1},\cdots,\hat{L}_{i_m})^T$ on $F_{i_1\cdots i_m}$ and 0 on $F\setminus F_{i_1\cdots i_m}$ up to isometry. Moreover, if the solution converges to the degenerated circle pattern metric $\hat{k}$, then $\hat{k}\in\mathbb{R}^{|F|}_{i_1\cdots i_m}$. 
\end{theorem}

\begin{theorem}\label{s5}
Given a closed topological surface $S$ with a cellular decomposition $\Sigma=(V,E,F)$, the weight $\Phi\in (0,\frac{\pi}{2})^{|E|}$ and the prescribed total geodesic curvature $\hat{L}\in\mathcal{L}_{i_1\cdots i_m} (1\le m\le |F|-1$, $1\le i_1<\cdots<i_m\le |F|)$ on the face set $F$. For any initial geodesic curvature $k(0)\in\mathbb{R}_{>0}^{|F|}$, the solution of the prescribed combinatorial Ricci flows (\ref{s2}) converges to the unique degenerated circle pattern metric with the total geodesic curvature $\hat{L}$ up to isometry. Moreover, if the solution converges to the degenerated circle pattern metric $\hat{k}$, then $\hat{k}\in\mathbb{R}^{|F|}_{i_1\cdots i_m}$. 
\end{theorem}

\begin{remark}

By Theorem \ref{s6} and Theorem \ref{s5}, for finding the degenerated circle pattern metric with the total geodesic curvature $\hat{L}\in\mathcal{L}_{i_1\cdots i_m}~(1\le m\le |F|-1, 1\le i_1<\cdots<i_m\le |F|)$, we have two methods. The first one is by constructing prescribed combinatorial Ricci flows (\ref{s3}) and the second one is by constructing prescribed combinatorial Ricci flows (\ref{s2}). Then the solutions of the two flows will converge to the same degenerated circle pattern metric with the total geodesic curvature $\hat{L}$ in different dimensions (see figure \ref{fig6}).  
\end{remark}

\begin{figure}[htbp]
\centering
\includegraphics[scale=0.6]{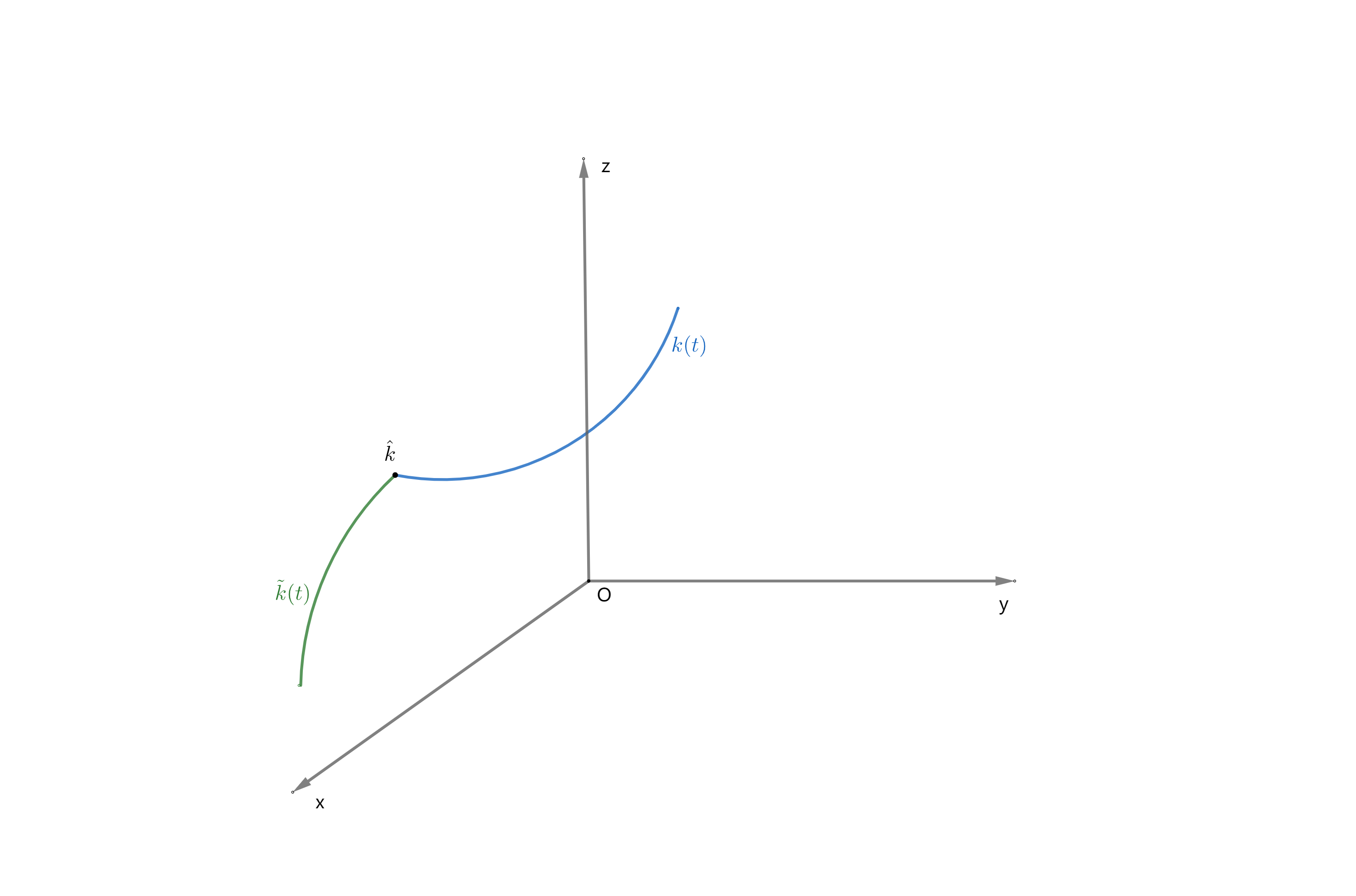}
\captionof{figure}{\small The solutions of the two flows converge to $\hat{k}$  }
\label{fig6}
\end{figure} 

\noindent\textbf{Organization} This paper is organized as follows. In section \ref{a1}, we define the circle patterns on spherical surfaces. In section \ref{a2}, we study the moduli space of spherical conical metrics. In section \ref{a3}, we construct the potential functions on bigons and define the moduli space of spherical bigons with given angle. In section \ref{a4}, we prove the existence and rigidity of circle pattern metrics for prescribed total geodesic curvatures. In section \ref{a5}, we prove the existence and rigidity of degenerated circle pattern metrics for prescribed total geodesic curvatures. In section \ref{a6}, we prove the existence and rigidity of spherical conical metrics for prescribed total geodesic curvatures. In section \ref{a7}, we study the convergence of prescribed combinatorial Ricci flows for circle pattern metrics. In section \ref{a8}, we study the convergence of prescribed combinatorial Ricci flows for degenerated circle pattern metrics.

\section{ Circle patterns on spherical surfaces}\label{a1}

Let $S$ be a topological surface, then $S$ is called a spherical surface if every
point $p$ in $S$ has a neighborhood $U$ such that there is a isometry from $U$ onto an open subset of $\mathbb{S}_{\alpha}^2$, where $\mathbb{S}_{\alpha}^2$ is a unit sphere with a cone angle $\alpha>0$. We also often call the spherical surface $S$ a surface with spherical conical metric.

By $D_{\alpha}(o,r)$ we denote the open disk of radius $r$ in $\mathbb{S}_{\alpha}^2$ centered at the cone point $o\in\mathbb{S}_{\alpha}^2$. In this paper, we only consider the disk $D_{\alpha}(o,r)$ with radius $r\in (0,\frac{\pi}{2}]$. By a bigon in a spherical surface $S$ we denote an open set isometric to the intersection of two disks of radii $r_1, r_2\in (0,\frac{\pi}{2}]$ in $\mathbb{S}_{\alpha}^2$ not containing
each other. Besides, a bigon does not contain any cone point. The angle of a bigon is the interior angle $\phi$ formed by its two sides.      
\begin{figure}[htbp]
\centering
\includegraphics[scale=0.3]{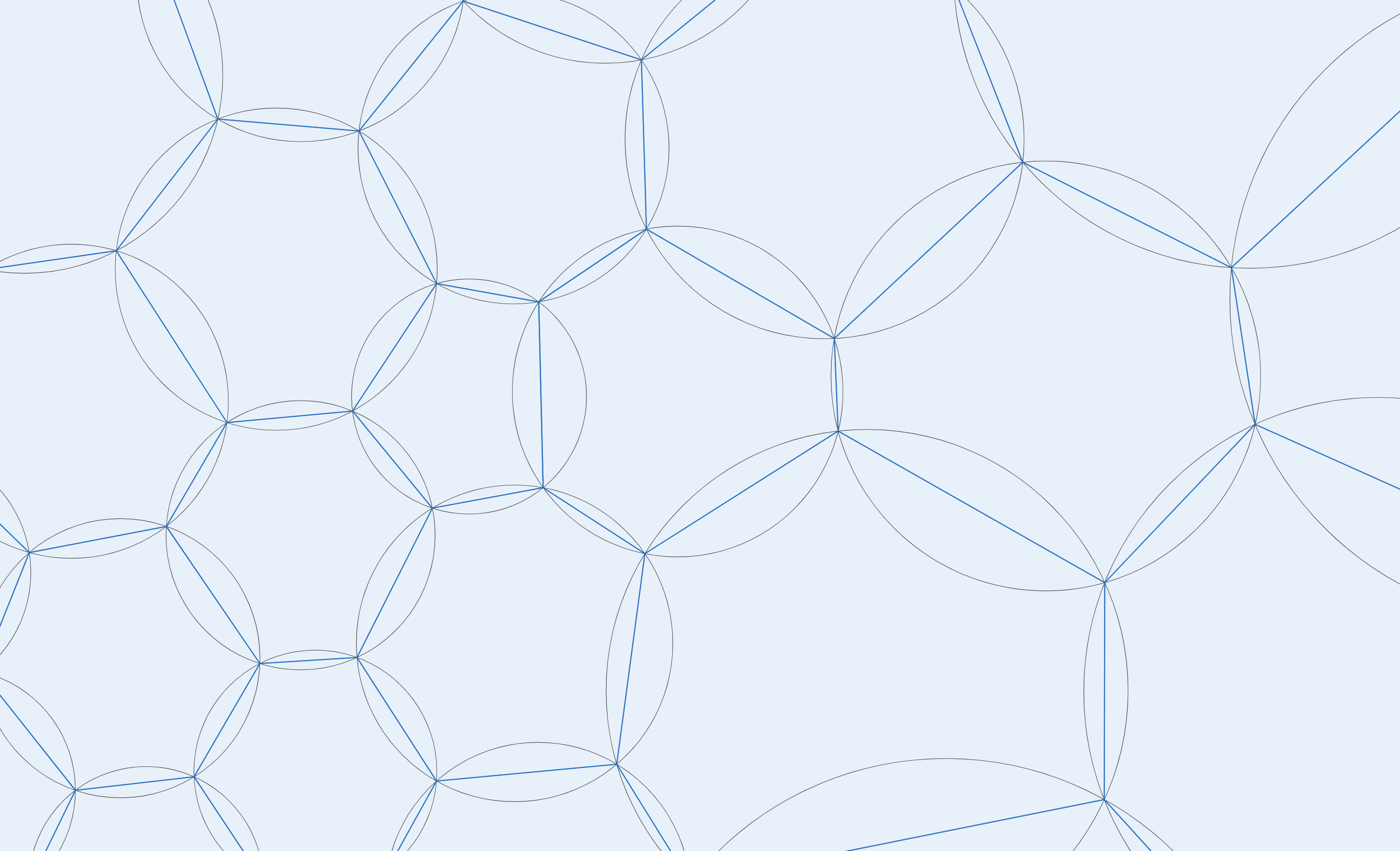}
\captionof{figure}{\small The circle pattern $\mathcal{P}$ and cellular decomposition $\Sigma$ on $S$ }
\label{fig7}
\end{figure}
Using similar analysis in proof of Proposition 2 in \cite{nie}, we have the following lemma:

\begin{lemma}\label{lemma6}
Given some disks $D_i=D_{\alpha_i}\left(o, r_i\right)~(1\le i\le n)$ as above with $\alpha_i>0, r_i \in (0, \frac{\pi}{2}]$. Let $\eta$ be a local isometry from the disjoint union $D_1 \sqcup \cdots \sqcup D_n$ to a closed spherical surface $\Gamma$ such that
\begin{enumerate}
    \item the set $\Gamma\setminus\eta(D_1 \sqcup \cdots \sqcup D_n)$ only has finitely many points;
    \item $\eta$ is at most 2 to 1;  
    \item the $\eta\left(D_i\right)~(1\le i\le n)$ do not contain each other.
\end{enumerate}   
By $\mathscr{B}$ we denote the set $\mathscr{B}:=\{p \in \Gamma \mid \eta^{-1}(p)~\text{has two points}\}$. Then $\eta^{-1}(\mathscr{B})$ is a disjoint union of bigons which fills up the boundary of the $D_i~(1\le i\le n)$.
\end{lemma}

For more details, we refer \cite{nie} to the readers. Then we can define the circle patterns on spherical surfaces in this paper.

\begin{definition}
A local isometry $\eta: D_1 \sqcup \cdots \sqcup D_n \rightarrow \Gamma$ as in Lemma \ref{lemma6} is called a circle pattern on $\Gamma$ and each connected component of the set $\mathscr{B} \subset \Gamma$ is called a bigon of the circle pattern. For simplicity, we also call $D_i~(1\le i\le n)$ an immersed disk in $\Gamma$ and call the set of immersed disks $\mathcal{P}=\left\{D_1, \cdots, D_n\right\}$ a circle pattern (so that each bigon is either an intersection component of two such disks $D_i, D_j$ or a "self-intersection component" of a single disk $D_i$).   
\end{definition}

Let $(G_{\mathcal{P}},\Phi)$ be the weighted graph on $\Gamma$ with the circle pattern $\mathcal{P}$ which satisfies
\begin{enumerate}
    \item the vertex set $V_{\mathcal{P}}$ is the complement of $D_1 \cup \cdots \cup D_n$ in $\Gamma$: $V_{\mathcal{P}}=\Gamma \backslash D_1 \cup \cdots \cup D_n$;
    \item the edges in edge set $E_{\mathcal{P}}$ are  1 to 1 correspondence with the bigons: given a bigon $B$, then the corresponding edge $e \in E_{\mathcal{P}}$ is a curve in $B$ joining the two endpoints of the bigon $B$;
    \item each edge $e \in E_{\mathcal{P}}$ is weighted by the angle $\Phi(e)\in(0, \pi)$ of the corresponding bigon;
    \item each face $f\in F_{\mathcal{P}}$ corresponds to a disk $D_i\in\mathcal{P}$: $f$ is contained in $\bar{D}_i$ and each vertex on $\partial f$ is in $\partial D_i$. 
\end{enumerate} 

\begin{remark}
Given a closed spherical surface $S$ with a circle pattern $\mathcal{P}$, then the weighted graph $(G_{\mathcal{P}},\Phi)$ gives a cellular decomposition $\Sigma$ of the surface $S$ (see figure \ref{fig7}).
\end{remark}

 \begin{figure}[htbp]
\centering
\includegraphics[scale=1.2]{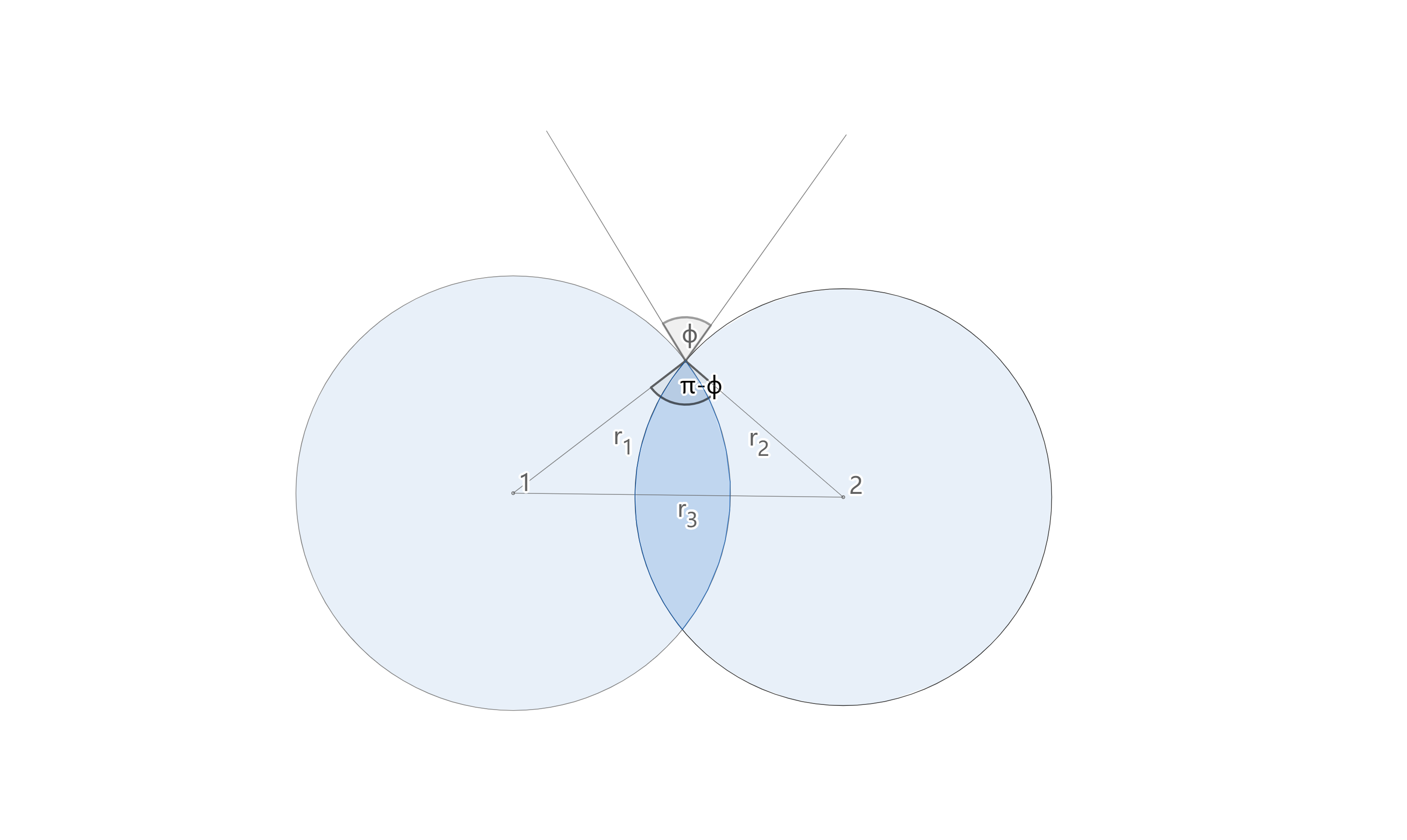}
\captionof{figure}{\small The bigon of angle $\phi$ on $\mathbb{S}^2$}
\label{fig1}
\end{figure} 

\section{The moduli space of spherical conical metrics}\label{a2}

\begin{lemma}\label{lemma1}

For any $r_1,r_2\in (0,\frac{\pi}{2}]$ and any $\phi\in (0,\frac{\pi}{2})$, there exists a bigon of angle $\phi$ which is formed by the disk $D_i\subset \mathbb{S}^2$, where the $D_i$ has radius $r_i$, $i=1,2$. Moreover, the bigon is unique up to isometry.
\end{lemma}

\begin{proof}
    First consider the case $r_1\le r_2$. Let $D_i\subset \mathbb{S}^2$ be the disk which has radius $r_i$, $i=1,2$. We suppose that the distance between the centers of $D_1$ and $D_2$ is $r_3$, the intersection angle of $D_1$ and $D_2$ is $\phi$. By the law of cosines in spherical geometry, we have the following identity, i.e.
    $$
    \cos r_3=\cos r_1 \cos r_2+\sin r_1 \sin r_2 \cos (\pi-\phi)=\cos r_1 \cos r_2-\sin r_1 \sin r_2 \cos \phi.
    $$
    
    For $r_1,r_2\in (0, \frac{\pi}{2}]$, $\phi\in (0,\frac{\pi}{2})$, we have
    $$ 
    \cos(r_1+r_2)<\cos r_3=\cos r_1 \cos r_2-\sin r_1 \sin r_2 \cos \phi<\cos(r_1+r_2).
    $$

   Since $0\le r_1+r_2, r_2-r_1\le \pi$, $0<r_3<r_1+r_2\le \pi$, we know that the $r_3$ is unique and $r_3\in (r_2-r_1, r_1+r_2)$. The $D_1\cap D_2$ is the bigon of angle $\phi$ which is formed by the disks $D_1, D_2\subset \mathbb{S}^2$ as shown in Figure \ref{fig1}. Since $r_3$ is unique, the bigon is unique up to isometry.
   
   For the case $r_1>r_2$, it is similar to the analysis above. This completes the proof. 
\end{proof}

\begin{remark}
For the disks $D_1, D_2\subset\mathbb{S}^2$ with radii $r_1, r_2\in (0,\frac{\pi}{2}]$ and the intersection angle $\phi\in (0,\frac{\pi}{2})$, $\bar{D}_1$ do not contain the center of $D_2$ and $\bar{D}_2$ do not contain the center of $D_1$. 
\end{remark}

   Given a cellular decomposition $\Sigma$ with the weight $\Phi\in (0,\frac{\pi}{2})^{|E|}$ on the closed topological surface $S$, then we can define the moduli space $\mathcal{T}$ of all spherical conical metrics with circle patterns on $S$ which realize the weighted graph $(\Sigma,\Phi)$, i.e.
   $$
   \mathcal{T}:=\left\{\begin{array}{l|l}
(\mu, \mathcal{P}) & \begin{array}{l}
\mu \text { is a spherical conical metric on } S ; \\
\mathcal{P} \text { is a circle pattern on }(S, \mu) \text { such that } G_{\mathcal{P}} \\
\text { is isotopic to}~\Sigma~\text{and realizes the weight}~\Phi
\end{array}
\end{array}\right\} / \sim,
$$
where $(\mu_1, \mathcal{P}_1)\sim (\mu_2, \mathcal{P}_2)$ if and only if there exists an isometry $f: S\to S$ such that $f^*\mu_2=\mu_1, f^*\mathcal{P}_2=\mathcal{P}_1$ and $f$ is isotopic to the identity map. Then we can define a map $\kappa$ from $(0,\frac{\pi}{2}]^{|F|}$ to $\mathcal{T}$. By Lemma \ref{lemma1}, the map $\kappa$ is well-defined. Besides, we have that $(0,\frac{\pi}{2}]^{|F|}\cong \mathcal{T}$, i.e.  
\begin{proposition}\label{pro1}
    Given the weighted graph $(\Sigma,\Phi)$ on $S$, then the map $\kappa$
     $$
\begin{array}{cccc}
\kappa: (0,\frac{\pi}{2}]^{|F|} &\longrightarrow &\mathcal{T}\\
\mathbf{r}=(r_1,\cdots,r_{|F|})^T&\longmapsto & [(\mu,\mathcal{P})]\\
\end{array}
$$
is bijective.
\end{proposition}

\begin{proof}
If $\kappa(\mathbf{r}_1)=\kappa(\mathbf{r}_2)$, let $\kappa(\mathbf{r}_i)=[(\mu_i,\mathcal{P}_i)], i=1,2$. Then we know that $(\mu_1,\mathcal{P}_1)\sim (\mu_2,\mathcal{P}_2)$. By definition, the radii of the corresponding disks in circle patterns $\mathcal{P}_1$ and $\mathcal{P}_2$ are equal, i.e. $\mathbf{r}_1=\mathbf{r}_2$. The map $\kappa$ is injective. 

For any $[(\mu,\mathcal{P})]\in\mathcal{T}$, there exists a corresponding radii $\mathbf{r}\in (0,\frac{\pi}{2}]^{|F|}$. Given the weight $\phi$, by Lemma \ref{lemma1}, the local spherical metric generated by $\mathbf{r}$ is isometric to the corresponding local metric of $(\mu,\mathcal{P})$. Hence the spherical metric generated by $\mathbf{r}$ is isometric to the metric of $(\mu,\mathcal{P})$, i.e. $\kappa(\mathbf{r})=[(\mu,\mathcal{P})]$. The map $\kappa$ is surjective.       
\end{proof}

 By Proposition \ref{pro1}, we know that $(0,\frac{\pi}{2}]^{|F|}\cong \mathcal{T}\cong \mathbb{R}^{|F|}_{\ge 0}$.

\section{Potential functions on bigons}\label{a3}

Given $\phi\in (0, \frac{\pi}{2})$, we consider two disks $D_1, D_2\subset \mathbb{S}^2$. Suppose that the radius of disk $D_i$ is $r_i\in (0, \frac{\pi}{2}]$, $i=1,2$. Besides, let the intersection angle of $D_1$ and $D_2$ be $\phi$. By $B$ we denote the bigon which is formed by the intersection of disks $D_1$ and $D_2$. Let $\ell_i$ denote the length of the sides of bigon $B$, $i=1,2$. By $k_i$ we denote the geodesic curvature of $\partial D_i$, i.e. $k_i=\cot r_i$, $i=1,2$.
By $L_i$ we denote the total geodesic curvature of the sides of bigon $B$, i.e. $L_i=\ell_ik_i$, $i=1,2$ (see Figure \ref{fig5}).

\begin{figure}[htbp]
\centering
\includegraphics[scale=1.4]{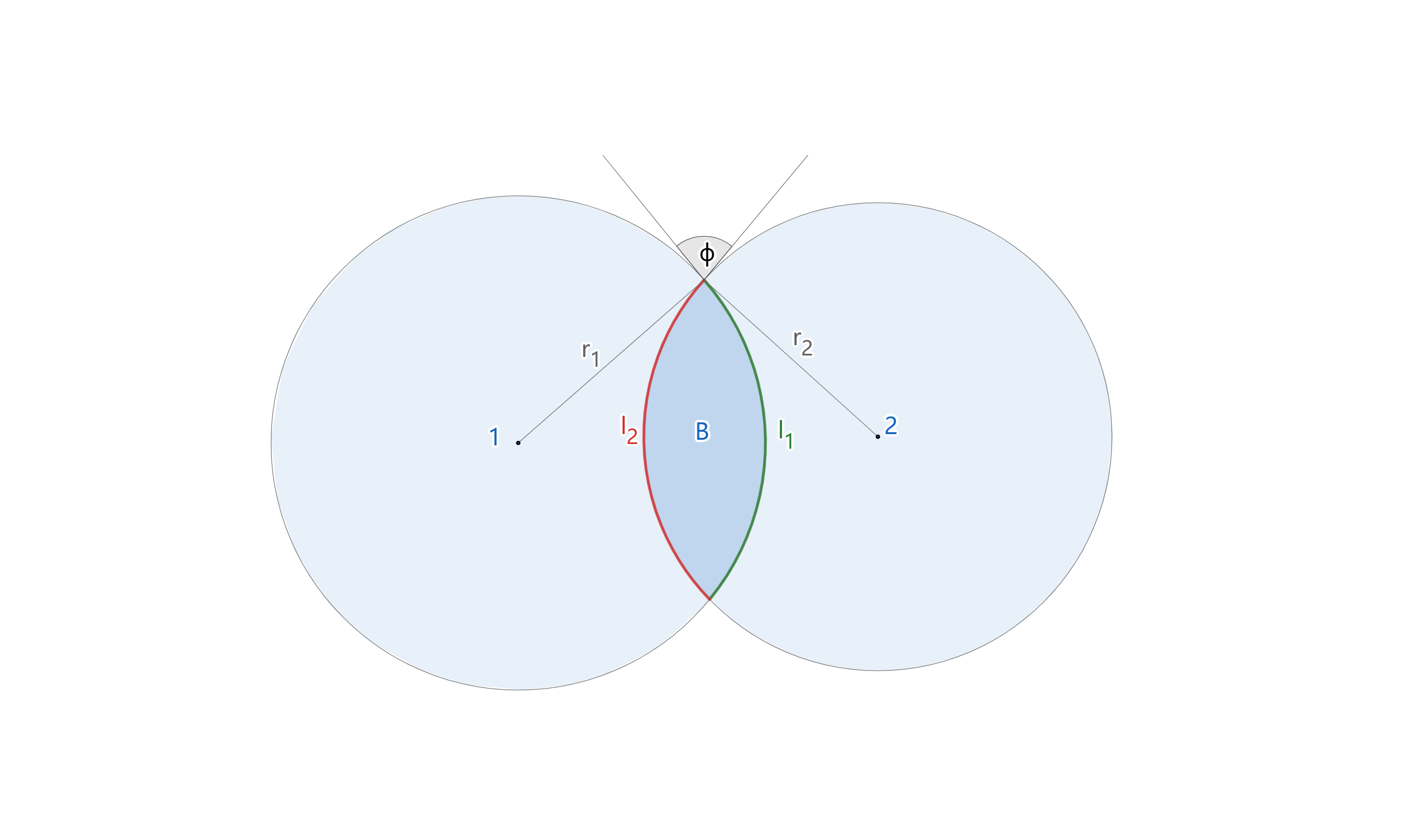}
\captionof{figure}{\small The disk $D_1, D_2\subset\mathbb{S}^2$ }
\label{fig5}
\end{figure} 

\begin{lemma}\label{lemma2}
    Given $\phi\in (0,\frac{\pi}{2})$, then the 1-form $\omega_{\phi}=\ell_1dk_1+\ell_2dk_2$ is closed on $\mathbb{R}^2_{>0}$, $\mathbb{R}_{>0}\times \{k_2=0\}$ and $\{k_1=0\}\times \mathbb{R}_{>0}$. 
\end{lemma}
   
\begin{proof}
    By Lemma 8 in \cite{nie}, the 1-form $\omega_{\phi}$ is closed on $\mathbb{R}^2_{>0}$. On $\mathbb{R}_{>0}\times \{k_2=0\}$, $\omega_{\phi}=\ell_1dk_1$. Since $\ell_1$ is a smooth function with respect to $k_1$, $d\omega_{\phi}=0$, $\omega_{\phi}$ is closed. Similarly, we know that $\omega_{\phi}$ is also closed on $\{k_1=0\}\times \mathbb{R}_{>0}$.      
\end{proof}
Using change of variables $K_i=\ln k_i$, $i=1,2$, by Lemma \ref{lemma2} the 1-form $\omega_{\phi}=L_1dK_1+L_2dK_2$ is closed on $\mathbb{R}^2$, $\omega_{\phi}=L_1dK_1$ is closed on $\mathbb{R}\times \{-\infty\}$ and $\omega_{\phi}=L_2dK_2$ is closed on $\{-\infty\}\times\mathbb{R}$. On $\mathbb{R}^2$, we can define a potential function
$$
\Lambda_{\phi}(K_1,K_2):=\int_0^{\left(K_1, K_2\right)} \omega_\phi.
$$
On $\mathbb{R}\times \{-\infty\}$, we define a potential function
$$
\Lambda_{\phi}(K_1):=\int_0^{K_1} \omega_\phi.
$$
On $\{-\infty\}\times\mathbb{R}$, we define a potential function
$$
\Lambda_{\phi}(K_2):=\int_0^{K_2} \omega_\phi.
$$
These three functions are well-defined. Besides, we have that $\nabla\Lambda_{\phi}(K_1,K_2)=(L_1,L_2)^T$, $\Lambda^{'}_{\phi}(K_1)=L_1$ and $\Lambda^{'}_{\phi}(K_2)=L_2$. 

\begin{lemma}\label{lemma3}
We have that $\frac{\partial L_1}{\partial K_1}>0$ on $\mathbb{R}\times \{-\infty\}$, $\frac{\partial L_2}{\partial K_2}>0$ on $\{-\infty\}\times\mathbb{R}$ and the matrix    
$$
\left(\begin{array}{ll}
\frac{\partial L_1}{\partial K_1} & \frac{\partial L_1}{\partial K_2} \\
\frac{\partial L_2}{\partial K_1} & \frac{\partial L_2}{\partial K_2}
\end{array}\right) 
$$
is positive definite on $\mathbb{R}^2$.
\end{lemma}

\begin{proof}
    By Lemma 9 in \cite{nie}, we know that on $\mathbb{R}^2$, the matrix
    $$
\left(\begin{array}{ll}
\frac{\partial L_1}{\partial K_1} & \frac{\partial L_1}{\partial K_2} \\
\frac{\partial L_2}{\partial K_1} & \frac{\partial L_2}{\partial K_2}
\end{array}\right) 
$$
is positive definite. 

On $\mathbb{R}\times \{-\infty\}$, since $k_2=0$, $L_2=0$. By Gauss-Bonnet formula, we have
$$
\text{Area}(B)=2\phi-L_1.
$$
By Lemma 9 in \cite{nie}, we know that $\frac{\partial \text{Area}(B)}{\partial K_1}<0$, then $\frac{\partial L_1}{\partial K_1}>0$.
Similarly, we know that $\frac{\partial L_2}{\partial K_2}>0$ on $\{-\infty\}\times\mathbb{R}$. 
\end{proof}

By Lemma \ref{lemma3}, $\Lambda_{\phi}(K_1,K_2)$ is a strictly convex function on $\mathbb{R}^2$, $\Lambda_{\phi}(K_1)$ is a strictly convex function on $\mathbb{R}\times \{-\infty\}$ and $\Lambda_{\phi}(K_2)$ is a strictly convex function on $\{-\infty\}\times\mathbb{R}$.

We can define the moduli space $\mathcal{B}_{\phi}$ of spherical bigons with given angle $\phi\in (0,\frac{\pi}{2})$, i.e.
$$
   \mathcal{B}_{\phi}:=\left\{\begin{array}{l|l}
B & \begin{array}{l}
B \text { is a spherical bigon which is formed by} \\
\text {the intersection of two disks}~D_1,D_2\subset\mathbb{S}^2~\text{with} \\
\text { given intersection angle}~\phi\in (0,\frac{\pi}{2}),~\text{where}\\
\text{the radii}~r_1, r_2~\text{ of the two disks are in }~(0,\frac{\pi}{2}]  \\
\end{array}
\end{array}\right\} / \sim,
$$   
where $B_1\sim B_2$ if and only if $B_1$ and $B_2$ are isomorphic.

Then we can construct a map $\gamma$ from $\mathcal{B}_{\phi}$ to $(0, \frac{\pi}{2}]^2$, i.e.
$$
\begin{aligned}
\gamma: \mathcal{B}_{\phi} & \longrightarrow (0, \frac{\pi}{2}]^2 \\
{[B] } & \longmapsto\left(r_1, r_2\right)
\end{aligned},
$$
The map $\gamma$ is well-defined. Besides, by Lemma \ref{lemma1}, $\gamma$ is a bijective map, $\mathcal{B}_{\phi}\cong  (0, \frac{\pi}{2}]^2$.   

Then we define some subspaces of the moduli space $\mathcal{B}_{\phi}$, i.e.
$$
   \mathcal{B}^0_{\phi}:=\left\{\begin{array}{l|l}
B & \begin{array}{l}
B \text { is a spherical bigon which is formed by} \\
\text {the intersection of two disks}~D_1,D_2\subset\mathbb{S}^2~\text{with} \\
\text { given intersection angle}~\phi\in (0,\frac{\pi}{2}),~\text{where}\\
\text{the radii}~r_1, r_2~\text{ of the two disks are}~\frac{\pi}{2}  \\
\end{array}
\end{array}\right\} / \sim,
$$  
$$
   \mathcal{B}^1_{\phi}:=\left\{\begin{array}{l|l}
B & \begin{array}{l}
B \text { is a spherical bigon which is formed by} \\
\text {the intersection of two disks }~D_1,D_2\subset\mathbb{S}^2~\text{with} \\
\text { given intersection angle}~\phi\in (0,\frac{\pi}{2}),~\text{where}\\
\text{the radii}~r_1, r_2~\text{ of the two disks satisfy}\\
r_1\in (0,\frac{\pi}{2}), r_2=\frac{\pi}{2}  \\
\end{array}
\end{array}\right\} / \sim,
$$  
$$
   \mathcal{B}^2_{\phi}:=\left\{\begin{array}{l|l}
B & \begin{array}{l}
B \text { is a spherical bigon which is formed by} \\
\text {the intersection of two disks }~D_1,D_2\subset\mathbb{S}^2~\text{with} \\
\text { given intersection angle}~\phi\in (0,\frac{\pi}{2}),~\text{where}\\
\text{the radii}~r_1, r_2~\text{ of the two disks satisfy}\\
r_1=\frac{\pi}{2}, r_2\in (0,\frac{\pi}{2})  \\
\end{array}
\end{array}\right\} / \sim,
$$     
$$
   \mathcal{B}^3_{\phi}:=\left\{\begin{array}{l|l}
B & \begin{array}{l}
B \text { is a spherical bigon which is formed by} \\
\text {the intersection of two disks }~D_1,D_2\subset\mathbb{S}^2~\text{with} \\
\text { given intersection angle}~\phi\in (0,\frac{\pi}{2}),~\text{where}\\
\text{the radii}~r_1, r_2~\text{ of the two disks are in }~(0,\frac{\pi}{2})  \\
\end{array}
\end{array}\right\} / \sim,
$$      
It is easy to know that the subspace $\mathcal{B}^0_{\phi}, \mathcal{B}^1_{\phi}, \mathcal{B}^2_{\phi}$ and $\mathcal{B}^3_{\phi}$ are mutually disjoint. Besides, we have that $\mathcal{B}_{\phi}=\mathcal{B}^0_{\phi} \sqcup \mathcal{B}^1_{\phi} \sqcup \mathcal{B}^2_{\phi} \sqcup \mathcal{B}^3_{\phi}.$ Similarly, we have that $\mathcal{B}^0_{\phi}\cong \{(\frac{\pi}{2}, \frac{\pi}{2})\}, \mathcal{B}^1_{\phi}\cong (0, \frac{\pi}{2})\times \{\frac{\pi}{2}\}, \mathcal{B}^2_{\phi}\cong \{\frac{\pi}{2}\}\times (0, \frac{\pi}{2})$ and $\mathcal{B}^3_{\phi}\cong (0, \frac{\pi}{2})^2$.        

\begin{lemma}\label{lemma4}
    We have that $\mathbb{R}\times \{-\infty\}\cong \Delta_1, \{-\infty\}\times \mathbb{R}\cong \Delta_2$, where $\Delta_1=\{(x,0)~|~0<x<2\phi\}, \Delta_2=\{(0,y)~|~0<y<2\phi\}$.  
\end{lemma}
    \begin{proof}
        On $\mathbb{R}\times \{-\infty\}$, by Lemma \ref{lemma3}, it is easy to know that $\Lambda^{'}_{\phi}(K_1)$ is a embedding map. For the map $\Lambda^{'}_{\phi}(K_1)$, we have 
        $$
\begin{aligned}
\Lambda^{'}_{\phi}(K_1): \mathbb{R} & \longrightarrow \mathbb{R}_{>0} \\
K_1 & \longmapsto L_1(K_1).
\end{aligned}
$$
On $\mathbb{R}\times \{-\infty\}$, $L_2=0$. By Gauss-Bonnet formula, we have that $\text{Area}(B)=2\phi-L_1>0$, i.e. $0<L_1<2\phi$. Hence the image set of $\Lambda^{'}_{\phi}(K_1)$ is contained in $\tilde{\Delta}_1$, where $\tilde{\Delta}_1=\{x~|~0<x<2\phi\}$. Then we will show that $0$ and $2\phi$ are the limit points of the map $\Lambda^{'}_{\phi}(K_1)$. 

We can choose a sequence $\{n\}_{n\in \mathbb{N}_+}$, then $k_1(n)=e^n\to +\infty~(n\to +\infty), r_1(n)=\arccot k_1(n)\to 0 ~(n\to +\infty)$. Hence we have that $\text{Area}(B(n))\to 0~(n\to +\infty)$. On $\mathbb{R}\times \{-\infty\}$, $L_2(n)=0$. By Gauss-Bonnet formula, we have that $\text{Area}(B(n))=2\phi-L_1(n)\to 0~(n\to +\infty)$, i.e. $L_1(n)\to 2\phi~(n\to +\infty)$.    

We choose another sequence $\{-n\}_{n\in \mathbb{N}_+}$, then $k_1(-n)=e^{-n}$. Since $\ell_1(-n)\le \pi$, then $0\le L_1(-n)=\ell_1(-n)k_1(-n)\le \pi e^{-n}\to 0~(n\to +\infty)$. Hence we have that $L_1(-n)\to 0~(n\to +\infty)$.      

Since $0$ and $2\phi$ are the limit points of the map $\Lambda^{'}_{\phi}(K_1)$ and $\Lambda^{'}_{\phi}(K_1)$ is a embedding map, it is easy to know that $\Lambda^{'}_{\phi}(K_1)$ is homeomorphism from $\mathbb{R}$ to $\tilde{\Delta}_1$, i.e. $\mathbb{R}\times \{-\infty\}\cong \Delta_1$.

Using the above same argument, we know that $\{-\infty\}\times \mathbb{R}\cong \Delta_2$. 
\end{proof}
By Lemma \ref{lemma4}, we have that $\mathcal{B}^1_{\phi}\cong (0, \frac{\pi}{2})\times \{\frac{\pi}{2}\}\cong \mathbb{R}\times \{-\infty\}\cong \Delta_1$ and $\mathcal{B}^2_{\phi}\cong  \{\frac{\pi}{2}\}\times (0, \frac{\pi}{2})\cong \{-\infty\}\times \mathbb{R}\cong \Delta_2$. By Lemma 11 in \cite{nie}, $\mathcal{B}^3_{\phi}\cong (0,\frac{\pi}{2})^2\cong \mathbb{R}^2\cong\Delta_3$, where $\Delta_3=\{(x,y)\in\mathbb{R}^2_{>0}~|~x+y<2\phi\}$. Besides, we have $\mathcal{B}^0_{\phi}\cong \{(\frac{\pi}{2},\frac{\pi}{2})\}\cong\{(0,0)\}$. 

\section{Existence and rigidity of circle pattern metrics for prescribed total geodesic curvatures}\label{a4} 

Given a closed topological surface $S$ and a cellular decomposition $\Sigma=(V,E,F)$ of $S$. For simplicity, we can suppose that $F=\{1,\cdots,|F|\}$. By $L_{f,e}$ we denote the total geodesic curvature of an arc $C_{f,e}$, where the arc $C_{f,e}$ is the side of the bigon corresponding to 2-cell $f\in F$ and 1-cell $e\in E$. By $B_e$ we denote the bigon corresponding to $e$ (see Figure \ref{fig4}).   
\begin{figure}[htbp]
\centering
\includegraphics[scale=0.8]{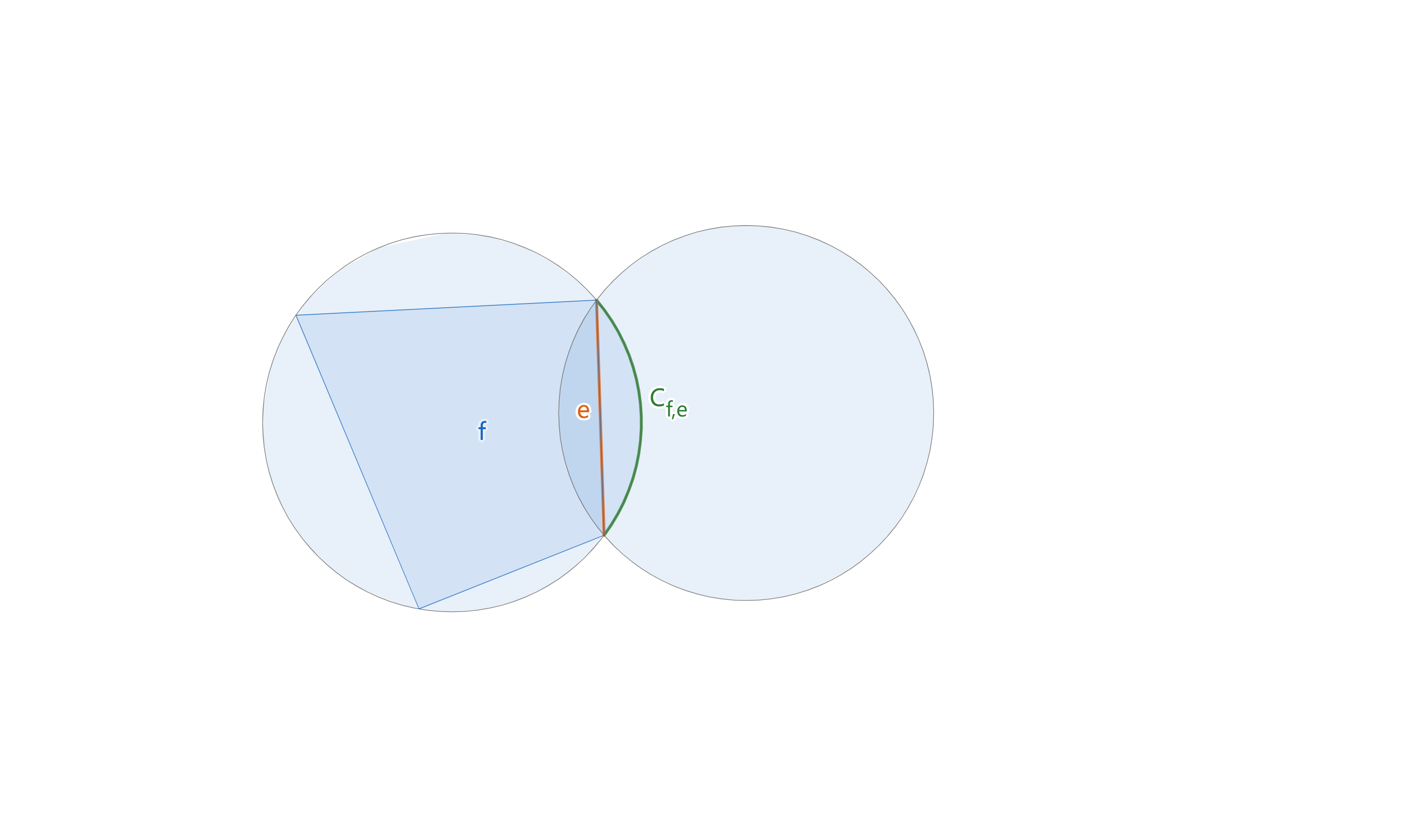}
\captionof{figure}{\small The position of $f,e,B_e$ and $C_{f,e}$}
\label{fig4}
\end{figure} 

Then we can define the total geodesic curvature $L_f$ of a 2-cell $f$, i.e. $L_f:=\sum_{e\in E_f}L_{f,e}$, where $E_f$ is the set of 1-cells contained in $\partial f$.

Given the weight $\Phi\in (0,\frac{\pi}{2})^{|E|}$, for $\mathbf{r}=(r_1,\cdots,r_{|F|})\in (0,\frac{\pi}{2})^{|F|}$, using change of variables $K_i=\ln \cot r_i$, $i=1,\cdots,|F|$, we consider the new variable $K=(K_1,\cdots,K_{|F|})\in \mathbb{R}^{|F|}$.

We can define a potential function 
$$
\Lambda(K):=\sum_{e\in E}\Lambda_{\Phi(e)}(K_{f_1(e)},K_{f_2(e)})
$$
on $\mathbb{R}^{|F|}$, where $f_1(e)$ and $f_2(e)$ are the 2-cells on two sides of 1-cell $e$. Besides, we have 
\begin{align}\label{formula1}
\frac{\partial \Lambda(K)}{\partial K_i}=\sum_{e\in E_i}\frac{\partial \Lambda_{\Phi(e)}(K_{f_1(e)},K_{f_2(e)})}{\partial K_i}=\sum_{e\in E_i}L_{i,e}=L_i>0,~1\le i\le |F|.
\end{align}
Hence the Hessian of $\Lambda$ is a  Jacobi matrix, i.e.
\[\text{Hess}~\Lambda=M=\begin{pmatrix}
 	 		\frac{\partial L_1}{\partial K_1}&\cdots&\frac{\partial L_{1}}{\partial K_{\vert F\vert}}\\
 	 		\vdots&\ddots&\vdots\\
 	 		\frac{\partial L_{\vert F\vert}}{\partial K_{1}}&\cdots&\frac{\partial L_{\vert F\vert}}{\partial K_{\vert F\vert}}\\
\end{pmatrix}.\]

\begin{proposition}\label{pro2}
    The Jacobi matrix $M$ is positive definite.
\end{proposition}

\begin{proof}
    By Lemma 9 in \cite{nie}, we have that
    $$
    \frac{\partial L_i}{\partial K_i}=\frac{\partial (\sum_{e\in E_i}L_{i,e})}{\partial K_i}=\sum_{e\in E_i}\frac{\partial L_{i,e}}{\partial K_i}>0, ~i=1,\cdots,|F|,
    $$
and
$$
    \frac{\partial L_i}{\partial K_j}=\frac{\partial (\sum_{e\in E_i}L_{i,e})}{\partial K_j}=\sum_{e\in E_i}\frac{\partial L_{i,e}}{\partial K_j}\le0, ~1\le i\ne j\le|F|.
    $$
Hence we obtain
$$
\left|\frac{\partial L_i}{\partial K_i}\right|-\sum_{j \neq i}\left|\frac{\partial L_j}{\partial K_i}\right|=\frac{\partial L_i}{\partial K_i}+\sum_{j \neq i} \frac{\partial L_j}{\partial K_i}=\frac{\partial (\sum_{f\in F}L_f)}{\partial K_i}.
$$
Besides, we have 
$$\sum_{f\in F}L_f=\sum_{f\in F}\sum_{e\in E_f}L_{f,e}=\sum_{e\in E}\sum_{f\in F_\{e\}}L_{f,e},$$ where $F_{\{e\}}$ is the set of 2-cells on two sides of 1-cell $e$. Then by Gauss-Bonnet formula, we have  
$$\sum_{f\in F_\{e\}}L_{f,e}=2\Phi(e)-\text{Area}(B_e).$$
Hence we obtain
$$
\sum_{f\in F}L_f=\sum_{e\in E}2\Phi(e)-\text{Area}(B_e).
$$
By Lemma 9 in \cite{nie}, we have
$$
\frac{\partial (\sum_{f\in F}L_f)}{\partial K_i}=-\sum_{e\in E}\frac{\partial (\text{Area}(B_e))}{\partial K_i}>0,
$$
Then we obtain 
$$
\left|\frac{\partial L_i}{\partial K_i}\right|-\sum_{j \neq i}\left|\frac{\partial L_j}{\partial K_i}\right|=\frac{\partial (\sum_{f\in F}L_f)}{\partial K_i}>0.
$$
Hence $M$ is a strictly diagonally dominant matrix with positive diagonal entries, i.e. $M$ is positive definite. 
\end{proof}

\begin{corollary}\label{coro1}
    The potential function $\Lambda$ is strictly convex on $\mathbb{R}^{|F|}$. 
\end{corollary}

\begin{proof}
    By Proposition \ref{pro2}, $\text{Hess}~\Lambda$ is positive definite, then $\Lambda$ is strictly convex on $\mathbb{R}^{|F|}$. 
\end{proof}

For a subset $F'\subset F$, by $E_{F'}$ we denote the set
$$
E_{F'}=\{e\in E~|~\exists f\in F', s.t. ~e\in E_f\}.
$$

Then we have the following theorem, i.e.

\begin{theorem}\label{thm1}
    $\nabla\Lambda$ is a homeomorphism from $\mathbb{R}^{|F|}$ to $\mathcal{L}_1$, where 
    $$
    \mathcal{L}_1=\{(L_1,\cdots,L_{|F|})^T\in \mathbb{R}^{|F|}_{>0}~|~\sum_{f\in F'}L_f<2\sum_{e\in E_{F'}}\Phi(e), \forall F'\subset F\}.
    $$
\end{theorem}

\begin{proof}
By (\ref{formula1}), we have 
$$
\begin{aligned}
\nabla\Lambda: \mathbb{R}^{|F|} & \longrightarrow \mathbb{R}^{|F|}_{>0}  \\
K=(K_1,\cdots,K_{|F|})^T & \mapsto L=(L_1,\cdots,L_{|F|})^T.
\end{aligned}
$$
For any subset $F'\subset F$, we obtain
$$
\sum_{f\in F'}L_f=\sum_{f\in F'}\sum_{e\in E_f}L_{f,e}=\sum_{e\in E_{F'}}\sum_{f\in F_{\{e\}}\cap F'}L_{f,e}.
$$
Then by Gauss-Bonnet formula, we have
$$
\text{Area}(B_e)=2\Phi(e)-\sum_{f\in F_{\{e\}}}L_{f,e}>0,
$$
i.e. $\sum_{f\in F_{\{e\}}}L_{f,e}<2\Phi(e)$. Hence we obtain
$$
\sum_{f\in F'}L_f=\sum_{e\in E_{F'}}\sum_{f\in F_{\{e\}}\cap F'}L_{f,e}\le\sum_{e\in E_{F'}}\sum_{f\in F_{\{e\}}}L_{f,e}<2\sum_{e\in E_{F'}}\Phi(e),
$$
hence $(L_1,\cdots,L_{|F|})^T\in\mathcal{L}_1$, $\nabla\Lambda$ is a map from $\mathbb{R}^{|F|}$ to $\mathcal{L}_1$. 

By Proposition \ref{pro2} and Corollary \ref{coro1}, $\nabla\Lambda$ is a embedding map. Hence we only need to show that the image set of $\nabla\Lambda$ is $\mathcal{L}_1$. Now we need to analysis the boundary of its image.     

Choose a point $a=(a_1,\cdots,a_{|F|})^T\in \partial\mathbb{R}^{|F|}$, we define two subset $W_1, W_2\subset F$, i.e.
$$
W_1=\{i\in F~|~a_i=+\infty\},~W_2=\{i\in F~|~a_i=-\infty\},
$$
we know that $W_1\ne\emptyset$ or $W_2\ne\emptyset$.

We choose a sequence $\{a^n\}_{n=1}^{+\infty}\subset \mathbb{R}^{|F|}$ such that $\lim_{n\to +\infty}a^n=a$, where $a^n=(a_1^n,\cdots,a_{|F|}^n)^T$. Then we have
$$
\nabla\Lambda(a^n)=L(a^n)\overset{\Delta}{=}L^n,~k_i^n\overset{\Delta}{=}\exp(a_i^n),~\forall n\ge 1.
$$
Now we need to show that $\{L^n\}_{n=1}^{+\infty}$ converges to the boundary of $\mathcal{L}_1$. 

If $W_1\ne\emptyset$, we know that for each $i\in W_1$, $k_i^n=\exp(a_i^n)\to +\infty ~(n\to +\infty)$, and $r_i^n=\arccot k_i^n\to 0 ~(n\to +\infty)$. Besides, note that
\begin{align}\label{for4}
\lim_{n\to +\infty}\sum_{f\in W_1}L_f^n=\lim_{n\to +\infty}\sum_{f\in W_1}\sum_{e\in E_f}L_{f,e}^n=\lim_{n\to +\infty}\sum_{e\in E_{W_1}}\sum_{f\in W_1\cap F_{\{e\}}}L_{f,e}^n.
\end{align}
Now we need to show
\begin{align}\label{for5}
\lim_{n\to +\infty}\sum_{f\in W_1\cap F_{\{e\}}}L_{f,e}^n=2\Phi(e).
\end{align}
Since $W_1\ne\emptyset$ and $e\in E_{W_1}$, the set $W_1\cap F_{\{e\}}$ has 1 or 2 elements.  

If $W_1\cap F_{\{e\}}$ has 1 elements, we can suppose that $W_1\cap F_{\{e\}}=\{h\}$ and $F_{\{e\}}=\{h,j\}$. By Gauss-Bonnet formula, we have 
\begin{align}\label{for2}
L_{h,e}^n+L_{j,e}^n=2\Phi(e)-\text{Area}(B_e^n).
\end{align}
Since $j\notin W_1$, we know that $a_j^n, k_j^n\nrightarrow +\infty, r_j^n\nrightarrow 0 ~(n\to +\infty)$. Since $h\in W_1$, we know that $r_h^n\to 0~(n\to +\infty)$. Then we have that $\ell_{j,e}^n, \text{Area}(B_e^n)\to 0~(n\to +\infty)$. Hence we obtain $L_{j,e}^n=\ell_{j,e}^nk_j^n\to 0~(n\to +\infty)$.

Using (\ref{for2}), we have
$$
\lim_{n\to +\infty}L_{h,e}^n=2\Phi(e)-\lim_{n\to +\infty}\text{Area}(B_e^n)=2\Phi(e),
$$
i.e. $\lim_{n\to +\infty}\sum_{f\in W_1\cap F_{\{e\}}}L_{f,e}^n=2\Phi(e)$.

If $W_1\cap F_{\{e\}}$ has 2 elements, we can suppose that $W_1\cap F_{\{e\}}=\{m,s\}$, then $F_{\{e\}}=\{m,s\}$. By Gauss-Bonnet formula, we have 
\begin{align}\label{for3}
L_{m,e}^n+L_{s,e}^n=2\Phi(e)-\text{Area}(B_e^n).
\end{align}
Since $m,s\in W_1$, we know that $k_m^n, k_s^n\to +\infty~(n\to +\infty)$, $r_m^n, r_s^n\to 0~(n\to +\infty)$. Then we have $\text{Area}(B_e^n)\to 0~(n\to +\infty)$. Using (\ref{for3}), we obtain 
$$
\lim_{n\to +\infty}L_{m,e}^n+L_{s,e}^n=2\Phi(e)-\lim_{n\to +\infty}\text{Area}(B_e^n)=2\Phi(e),
$$
i.e. $\lim_{n\to +\infty}\sum_{f\in W_1\cap F_{\{e\}}}L_{f,e}^n=2\Phi(e)$. 

Hence if $W_1\ne\emptyset$, by (\ref{for4}) and (\ref{for5}), we have
$$
\lim_{n\to +\infty}\sum_{f\in W_1}L_f^n=\lim_{n\to +\infty}\sum_{e\in E_{W_1}}\sum_{f\in W_1\cap F_{\{e\}}}L_{f,e}^n=\sum_{e\in E_{W_1}}\lim_{n\to +\infty}\sum_{f\in W_1\cap F_{\{e\}}}L_{f,e}^n=2\sum_{e\in E_{W_1}}\Phi(e),
$$
i.e. $\{L^n\}_{n=1}^{+\infty}$ converges to the boundary of $\mathcal{L}_1$.

If $W_2\ne\emptyset$, then for each $i\in W_2$, $k_i^n=\exp(a_i^n)\to 0~(n\to +\infty)$. For any 1-cell $e\in E_i$, since $0\le \ell_{i,e}^n\le 2\pi$ and $k_i^n\to 0~(n\to +\infty)$, then $\lim_{n\to +\infty}L_{i,e}^n=\lim_{n\to +\infty}\ell_{i,e}^nk_i^n=0$. Hence we obtain 
$$
\lim_{n\to +\infty}L_i^n=\lim_{n\to +\infty}\sum_{e\in E_i}L_{i,e}^n=0,
$$
i.e. $\{L^n\}_{n=1}^{+\infty}$ converges to the boundary of $\mathcal{L}_1$. 

By Brouwer's Theorem on the Invariance of Domain and the above analysis, we know that the image set of $\nabla\Lambda$ is $\mathcal{L}_1$. Hence $\nabla\Lambda$ is a homeomorphism from $\mathbb{R}^{|F|}$ to $\mathcal{L}_1$. 
\end{proof}
We can construct a map $\varsigma$ from $\mathbb{R}^{|F|}_{>0}$ to $\mathbb{R}^{|F|}$, i.e.
$$
\begin{array}{cccc}
\varsigma: \mathbb{R}^{|F|}_{>0} &\longrightarrow &\mathbb{R}^{|F|}\\
k=(k_1,\cdots,k_{|F|})^T&\longmapsto & K=(\ln k_1, \cdots, \ln k_{|F|})^T,\\
\end{array}
$$
the map $\varsigma$ is a homeomorphism. By Theorem \ref{thm1}, we know that the map $\mathcal{E}_1:=\nabla\Lambda\circ\varsigma$, i.e.
$$
\begin{array}{cccc}
\mathcal{E}_1: \mathbb{R}^{|F|}_{>0} &\longrightarrow &\mathcal{L}_1\\
k&\longmapsto & \mathcal{E}_1(k)=\nabla\Lambda(K)=L(K)\\
\end{array}
$$
is a homeomorphism from $\mathbb{R}^{|F|}_{>0}$ to $\mathcal{L}_1$. Then we completed the proof of Theorem \ref{thm10}. 

\section{Existence and rigidity of degenerated circle pattern metrics for prescribed total geodesic curvatures}\label{a5}

By $\mathbb{R}^{|F|}_{i_1\cdots i_m}$ we denote the set
$$
\mathbb{R}^{|F|}_{i_1\cdots i_m}=\{(0,\cdots,k_{i_1},0,\cdots,k_{i_m},\cdots,0)\in\mathbb{R}^{|F|}~|~k_{i_1},\cdots,k_{i_m}>0\}.
$$
By $\partial_i\mathbb{R}^{|F|}_{\ge 0}$ we denote the set
$$
\partial_i\mathbb{R}^{|F|}_{\ge 0}=\{(k_1,\cdots,k_{|F|})\in\mathbb{R}^{|F|}~|~i~\text{components}>0,~ |F|-i ~\text{components}=0\}.
$$
Then we have
$$
\partial_1\mathbb{R}^{|F|}_{\ge 0}=\bigcup_{1\le i\le |F|}\{(0,\cdots,k_i,\cdots,0)\in\mathbb{R}^{|F|}~|~k_i>0\}=\bigcup_{1\le i\le |F|}\mathbb{R}^{|F|}_i,
$$
$$
\partial_2\mathbb{R}^{|F|}_{\ge 0}=\bigcup_{1\le i<j\le |F|}\{(0,\cdots,k_i,\cdots,k_j,\cdots,0)\in\mathbb{R}^{|F|}~|~k_i,k_j>0\}=\bigcup_{1\le i<j\le |F|}\mathbb{R}^{|F|}_{ij},
$$
$$\cdots$$
$$
\partial_{|F|-1}\mathbb{R}^{|F|}_{\ge 0}=\bigcup_{1\le i_1<\cdots<i_{|F|-1}\le |F|}\{(k_i,\cdots,0,\cdots,k_{i_{|F|-1}})\in\mathbb{R}^{|F|}~|~k_{i_1},\cdots,k_{i_{|F|-1}}>0\}=\bigcup_{1\le i_1<\cdots<i_{|F|-1}\le |F|}\mathbb{R}^{|F|}_{i_1\cdots i_{|F|-1}}.
$$

For $1\le m\le |F|-1$, we have that 
$$\partial_m \mathbb{R}^{|F|}_{\ge 0}=\bigcup_{1\le i_1<\cdots<i_m\le |F|}\mathbb{R}^{|F|}_{i_1\cdots i_m}.$$

By $\bar{\partial}\mathbb{R}^{|F|}_{\ge 0}$ we denote the set 
$$
\bar{\partial}\mathbb{R}^{|F|}_{\ge 0}=\bigcup_{m=1}^{|F|-1}\partial_m \mathbb{R}^{|F|}_{\ge 0}\cup\{0\}.
$$

Then we consider the set $\mathbb{R}^{|F|}_{1\cdots m}$, i.e.
$$\mathbb{R}^{|F|}_{1\cdots m}=\{(k_1,\cdots,k_m,0,\cdots,0)\in\mathbb{R}^{|F|}~|~k_1,\cdots,k_m>0\}.$$ 
Given the weight $\Phi\in (0,\frac{\pi}{2})^{|E|}$, for $\mathbf{r}=(r_1,r_2,\cdots,r_m,\frac{\pi}{2},\cdots,\frac{\pi}{2})\in (0,\frac{\pi}{2})^{m}\times\{\frac{\pi}{2}\}^{|F|-m}$, using change of variables $K_i=\ln \cot r_i$, $i=1,\cdots,m$, we consider the new variable $K=(K_1,\cdots,K_{m})\in \mathbb{R}^{m}$.

We can define a potential function
$$
\Lambda_1(K):=\sum_{e\in E,f_1(e),f_2(e)\in F_m}\Lambda_{\Phi(e)}(K_{f_1(e)},K_{f_2(e)})+\sum_{e\in E,f_1(e)\in F_m,f_2(e)\notin F_m}\Lambda_{\Phi(e)}(K_{f_1(e)})
$$
on $\mathbb{R}^{m}$, where $F_m=\{1,2,\cdots,m\}$, $f_1(e)$ and $f_2(e)$ are the 2-cells on two sides of 1-cell $e$. Besides, note that at least one of these two sums exists.

For $1\le i\le m$, we have
\begin{align*}
\frac{\partial \Lambda_1(K)}{\partial K_i}&=\sum_{e\in E_i,f_1(e),f_2(e)\in F_m}\frac{\partial \Lambda_{\Phi(e)}(K_{f_1(e)},K_{f_2(e)})}{\partial K_i}+\sum_{e\in E_i,f_1(e)\in F_m,f_2(e)\notin F_m}\frac{\partial \Lambda_{\Phi(e)}(K_{f_1(e)})}{\partial K_i}\\
&=\sum_{e\in E_i,f_1(e),f_2(e)\in F_m}L_{i,e}(K_{f_1(e)},K_{f_2(e)})+\sum_{e\in E_i,f_1(e)\in F_m,f_2(e)\notin F_m}L_{i,e}(K_i).
\end{align*}
Since $i\in F_m$, then for any 1-cell $e\in E_i$, we know that $f_1(e), f_2(e)\in F_m$ or $f_1(e)\in F_m, f_2(e)\notin F_m$. Hence we have
\begin{align}\label{for6}
L_i=\sum_{e\in E_i}L_{i,e}=\sum_{e\in E_i,f_1(e), f_2(e)\in F_m}L_{i,e}(K_{f_1(e)},K_{f_2(e)})+\sum_{e\in E_i,f_1(e)\in F_m, f_2(e)\notin F_m}L_{i,e}(K_i)=\frac{\partial \Lambda_1(K)}{\partial K_i}.
\end{align}
Then the Hessian of $\Lambda_1$ is a Jacobi matrix, i.e.
\[\text{Hess}~\Lambda_1=M_1=\begin{pmatrix}
 	 		\frac{\partial L_1}{\partial K_1}&\cdots&\frac{\partial L_{1}}{\partial K_m}\\
 	 		\vdots&\ddots&\vdots\\
 	 		\frac{\partial L_{m}}{\partial K_{1}}&\cdots&\frac{\partial L_{m}}{\partial K_{m}}\\
\end{pmatrix}.\]

\begin{proposition}\label{pro3}
    The Jacobi matrix $M_1$ is positive definite.
\end{proposition}

\begin{proof}
For any $1\le i\le m$, by (\ref{for6}), we have
$$
\frac{\partial L_i}{\partial K_i}=\sum_{e\in E_i,f_1(e), f_2(e)\in F_m}\frac{\partial L_{i,e}(K_{f_1(e)},K_{f_2(e)})}{\partial K_i}+
\sum_{e\in E_i,f_1(e)\in F_m, f_2(e)\notin F_m}\frac{\partial L_{i,e}(K_i)}{\partial K_i}.
$$
By Lemma 9 in \cite{nie} and Lemma \ref{lemma3}, we have $\frac{\partial L_{i,e}(K_{f_1(e)},K_{f_2(e)})}{\partial K_i}>0, \frac{\partial L_{i,e}(K_i)}{\partial K_i}>0$. Besides, at least one of these two sums exists, then we obtain that $\frac{\partial L_i}{\partial K_i}>0$ for $1\le i\le m$. 

For $1\le i\ne j\le m$, by (\ref{for6}), we have 
$$
\frac{\partial L_i}{\partial K_j}=\sum_{e\in E_i,f_1(e), f_2(e)\in F_m}\frac{\partial L_{i,e}(K_{f_1(e)},K_{f_2(e)})}{\partial K_j}.
$$
For any 1-cell $e\in E_i$ such that $f_1(e),f_2(e)\in F_m$, we know that $i\in F_{\{e\}}$. If $j\in F_{\{e\}}$, by Lemma 9 in \cite{nie}, then we have 
$$\frac{\partial L_{i,e}(K_{f_1(e)},K_{f_2(e)})}{\partial K_j}=\frac{\partial L_{i,e}(K_i,K_j)}{\partial K_j}<0.$$      
If $j\notin F_{\{e\}}$, then we have
$$
\frac{\partial L_{i,e}(K_{f_1(e)},K_{f_2(e)})}{\partial K_j}=0.
$$
Then we obtain $\frac{\partial L_i}{\partial K_j}\le 0$ for $1\le i\ne j\le m$.

Hence for $1\le i\le m$, we have
\begin{align*}
\left|\frac{\partial L_i}{\partial K_i}\right|-\sum_{1\le j \ne i\le m}\left|\frac{\partial L_j}{\partial K_i}\right|&=\frac{\partial (\sum_{f\in F_m}L_f)}{\partial K_i}=\frac{\partial (\sum_{f\in F_m}\sum_{e\in E_f}L_{f,e})}{\partial K_i}=\frac{\partial (\sum_{e\in E_{F_{m}}}\sum_{f\in F_m\cap F_{\{e\}}}L_{f,e})}{\partial K_i}\\
&=\sum_{e\in E_{F_{m}}}\frac{\partial (\sum_{f\in F_m\cap F_{\{e\}}}L_{f,e})}{\partial K_i}=\sum_{e\in E_i}\frac{\partial (\sum_{f\in F_m\cap F_{\{e\}}}L_{f,e})}{\partial K_i}.
\end{align*}

For 1-cell $e\in E_i$, we can suppose that $F_{\{e\}}=\{i,h\}$. Then $F_m\cap F_{\{e\}}=\{i\}$ or $\{i,h\}$. 

If $F_m\cap F_{\{e\}}=\{i,h\}$, by Gauss-Bonnet formula, we have
$$
\sum_{f\in F_m\cap F_{\{e\}}}L_{f,e}=L_{i,e}+L_{h,e}=2\Phi(e)-\text{Area}(B_e).
$$
By Lemma 9 in \cite{nie}, we have 
$$
\frac{\partial (\sum_{f\in F_m\cap F_{\{e\}}}L_{f,e})}{\partial K_i}=-\frac{\partial \text{Area}(B_e)}{\partial K_i}>0.
$$
If $F_m\cap F_{\{e\}}=\{i\}$, then by Lemma \ref{lemma3}, we have 
$$
\frac{\partial (\sum_{f\in F_m\cap F_{\{e\}}}L_{f,e})}{\partial K_i}=\frac{\partial L_{i,e}(K_i)}{\partial K_i}>0.
$$
Then we obtain
$$
\left|\frac{\partial L_i}{\partial K_i}\right|-\sum_{1\le j \ne i\le m}\left|\frac{\partial L_j}{\partial K_i}\right|=\sum_{e\in E_i}\frac{\partial (\sum_{f\in F_m\cap F_{\{e\}}}L_{f,e})}{\partial K_i}>0.
$$
Hence $M_1$ is a strictly diagonally dominant matrix with positive diagonal entries, i.e. $M_1$ is positive definite. 
\end{proof}

\begin{corollary}\label{coro2}
    The potential function $\Lambda_1$ is strictly convex on $\mathbb{R}^{m}$. 
\end{corollary}

\begin{proof}
    By Proposition \ref{pro3}, $\text{Hess}~\Lambda_1$ is positive definite, then $\Lambda_1$ is strictly convex on $\mathbb{R}^{m}$. 
\end{proof}

\begin{theorem}\label{thm2}
    $\nabla\Lambda_1$ is a homeomorphism from $\mathbb{R}^{m}$ to $\tilde{\mathcal{L}}_{1\cdots m}$, where 
    $$
    \tilde{\mathcal{L}}_{1\cdots m}=\{(L_1,\cdots,L_{m})^T\in \mathbb{R}^{m}_{>0}~|~\sum_{f\in F_{m}'}L_f<2\sum_{e\in E_{F_{m}'}}\Phi(e), \forall F_{m}'\subset F_m\}.
    $$
\end{theorem}

\begin{proof}
By (\ref{for6}), we have 
$$
\begin{aligned}
\nabla\Lambda_1: \mathbb{R}^{m} & \longrightarrow \mathbb{R}^{m}_{>0}  \\
K=(K_1,\cdots,K_m)^T & \mapsto L=(L_1,\cdots,L_m)^T.
\end{aligned}
$$

For any subset $F_{m}'\subset F_{m}$, we obtain
$$
\sum_{f\in F_{m}'}L_f=\sum_{f\in F_{m}'}\sum_{e\in E_f}L_{f,e}=\sum_{e\in E_{F_{m}'}}\sum_{f\in F_{\{e\}}\cap F_{m}'}L_{f,e}.
$$
Then by Gauss-Bonnet formula, we have
$$
\text{Area}(B_e)=2\Phi(e)-\sum_{f\in F_{\{e\}}}L_{f,e}>0,
$$
i.e. $\sum_{f\in F_{\{e\}}}L_{f,e}<2\Phi(e)$. Hence we obtain
$$
\sum_{f\in F_{m}'}L_f=\sum_{e\in E_{F_{m}'}}\sum_{f\in F_{\{e\}}\cap F_{m}'}L_{f,e}\le\sum_{e\in E_{F_{m}'}}\sum_{f\in F_{\{e\}}}L_{f,e}<2\sum_{e\in E_{F_{m}'}}\Phi(e),
$$
hence $(L_1,\cdots,L_{m})^T\in\tilde{\mathcal{L}}_{1\cdots m}$, $\nabla\Lambda_1$ is a map from $\mathbb{R}^{m}$ to $\tilde{\mathcal{L}}_{1\cdots m}$.

By Proposition \ref{pro3} and Corollary \ref{coro2}, $\nabla\Lambda_1$ is a embedding map. Using the method in the proof of Theorem \ref{thm1}, we have a similar result that the image set of $\nabla\Lambda_1$ is $\tilde{\mathcal{L}}_{1\cdots m}$. Hence $\nabla\Lambda_1$ is a homeomorphism from $\mathbb{R}^{m}$ to $\tilde{\mathcal{L}}_{1\cdots m}$.    
\end{proof}
We can construct a map $\varsigma_1$ from $\mathbb{R}^{m}_{>0}$ to $\mathbb{R}^{m}$, i.e.
$$
\begin{array}{cccc}
\varsigma_1: \mathbb{R}^{m}_{>0} &\longrightarrow &\mathbb{R}^{m}\\
k=(k_1,\cdots,k_{m})^T&\longmapsto & K=(\ln k_1, \cdots, \ln k_{m})^T,\\
\end{array}
$$
the map $\varsigma_1$ is a homeomorphism. By Theorem \ref{thm2}, we know that the map $\tilde{\mathcal{E}}:=\nabla\Lambda_1\circ\varsigma_1$, i.e.
$$
\begin{array}{cccc}
\tilde{\mathcal{E}}: \mathbb{R}^{m}_{>0} &\longrightarrow &\tilde{\mathcal{L}}_{1\cdots m}\\
k&\longmapsto & \tilde{\mathcal{E}}(k)=\nabla\Lambda_1(K)=L(K)\\
\end{array}
$$
is a homeomorphism from $\mathbb{R}^{m}_{>0}$ to $\tilde{\mathcal{L}}_{1\cdots m}$. In other words, there exists a homeomorphism from $\mathbb{R}^{|F|}_{1\cdots m}$ to $\mathcal{L}_{1\cdots m}$, where
$$
\mathcal{L}_{1\cdots m}=\left\{\begin{array}{l|l}
(L_1,\cdots,L_{m},0,\cdots,0)^T\in \mathbb{R}^{m}_{>0}\times \{0\}^{|F|-m} & \begin{array}{l}
\sum_{f\in F_{m}'}L_f<2\sum_{e\in E_{F_{m}'}}\Phi(e), \forall F_{m}'\subset F_m
\end{array}
\end{array}\right\}.
$$

Similarly, for $1\le m\le |F|-1$, $1\le i_1<\cdots<i_m\le |F|$, by $F_{i_1\cdots i_m}$ we denote the set $F_{i_1\cdots i_m}=\{i_1,\cdots,i_m\}$. Then there exists a homeomorphism from $\mathbb{R}^{|F|}_{i_1\cdots i_m}$ to $\mathcal{L}_{i_1\cdots i_m}$, where
$$
\mathcal{L}_{i_1\cdots i_m}=\left\{\begin{array}{l|l}
(0,\cdots,L_{i_1},\cdots,0,\cdots,L_{i_m},0,\cdots,0)^T\in\mathbb{R}^{|F|} & \begin{array}{l}
\sum_{f\in F_{i_1\cdots i_m}'}L_f<2\sum_{e\in E_{F_{i_1\cdots i_m}'}}\Phi(e),\\
\forall F_{i_1\cdots i_m}'\subset F_{i_1\cdots i_m};L_{i_j}>0, 1\le j\le m
\end{array}
\end{array}\right\}.
$$    
Besides, we define the set $\tilde{\mathcal{L}}_{i_1\cdots i_m}$, i.e.
$$
\tilde{\mathcal{L}}_{i_1\cdots i_m}=\left\{\begin{array}{l|l}
(L_{i_1},\cdots,L_{i_m})^T\in\mathbb{R}_{>0}^{m} & \begin{array}{l}
\sum_{f\in F_{i_1\cdots i_m}'}L_f<2\sum_{e\in E_{F_{i_1\cdots i_m}'}}\Phi(e),~
\forall F_{i_1\cdots i_m}'\subset F_{i_1\cdots i_m}
\end{array}
\end{array}\right\}.
$$    
By $\bar{\partial}\mathcal{L}$ we denote the set 
$$\bar{\partial}\mathcal{L}=\bigcup_{1\le m\le |F|-1,1\le i_1<\cdots<i_m\le |F|}\mathcal{L}_{i_1\cdots i_m}\cup\{0\}.$$
Then by above argument, we completed the proof of Theorem \ref{thm11}.

\section{Existence and rigidity of spherical conical metrics for prescribed total geodesic curvatures}\label{a6}

By $\mathcal{E}_{i_1\cdots i_m}$ we denote the homeomorphism from $\mathbb{R}^{|F|}_{i_1\cdots i_m}$ to $\mathcal{L}_{i_1\cdots i_m}$ and by $\mathcal{L}$ we denote the set $\mathcal{L}=\mathcal{L}_1\cup\bar{\partial}\mathcal{L}$. Then we can define a map $\mathcal{E}$ from $\mathbb{R}^{|F|}_{\ge 0}$ to $\mathcal{L}$, i.e.
\begin{align}\mathcal{E}(k):=\left\{
\begin{array}{lc}
\mathcal{E}_1(k),& k\in\mathbb{R}^{|F|}_{>0}, \\\mathcal{E}_{i_1\cdots i_m}(k),&k\in\mathbb{R}^{|F|}_{i_1\cdots i_m},1\le m\le |F|-1, 1\le i_1<\cdots<i_m\le |F|,\\ 0,&k=0.
\end{array}
\label{for7}
\right.
\end{align}

\begin{theorem}\label{thm7}
    $\mathcal{E}$ is a homeomorphism from $\mathbb{R}^{|F|}_{\ge 0}$ to $\mathcal{L}$ such that the interior of $\mathbb{R}^{|F|}_{\ge 0}$ maps to the interior of $\mathcal{L}$ and $\bar{\partial}\mathbb{R}^{|F|}_{\ge 0}$ maps to $\bar{\partial}\mathcal{L}$. 
\end{theorem}
\begin{proof}
  Since $\mathcal{E}_1$, $\mathcal{E}_{i_1\cdots i_m} (1\le m\le |F|-1, 1\le i_1<\cdots<i_m\le |F|)$ are homeomorphisms, by (\ref{for7}), then $\mathcal{E}$ is a bijective map from $\mathbb{R}^{|F|}_{\ge 0}$ to $\mathcal{L}$. We only need to show that $\mathcal{E}$ and $\mathcal{E}^{-1}$ are continuous from interior to boundary.    

We choose a sequence $\{k^n\}_{n=1}^{+\infty}\subset\mathbb{R}^{|F|}_{>0}$ such that $\lim_{n\to +\infty}k^n=0$, where $k^n=(k_1^n,\cdots,k_{|F|}^n)^T$. Then we have $\mathcal{E}_1(k^n)=(L_1(k^n),\cdots,L_{|F|}(k^n))^T$, where $L_i(k^n)=\sum_{e\in E_i}\ell_{i,e}^n k_i^n$, $1\le i\le |F|$. Since $0\le \ell_{i,e}^n\le \pi, k_i^n\to 0~(n\to +\infty)$, then we have $\ell_{i,e}^n k_i^n\to 0~(n\to +\infty)$, $L_i(k^n)=\sum_{e\in E_i}\ell_{i,e}^n k_i^n\to 0~(n\to +\infty)$, $1\le i\le |F|$. Hence we obtain $\mathcal{E}_1(k^n)\to 0~(n\to +\infty)$.

For any point $k\in\mathbb{R}_{i_1\cdots i_m}^{|F|}$, $1\le m\le |F|-1$, $1\le i_1<\cdots<i_m\le |F|$, we choose another sequence $\{k^n\}_{n=1}^{+\infty}\subset\mathbb{R}^{|F|}_{>0}$ such that $\lim_{n\to +\infty}k^n=k$.
Then we need to show that $\mathcal{E}_1(k^n)\to \mathcal{E}_{i_1\cdots i_m}(k)~(n\to +\infty)$. For simplicity, we only show that $\mathcal{E}_1(k^n)\to \mathcal{E}_{1\cdots m}(k)=(L_1(k),\cdots,L_m(k),0,\cdots,0)~(n\to +\infty)$, the others use the same method.

We can suppose that $k=(k_1,\cdots,k_m,0,\cdots,0)^T$, since $k_i^n\to 0, ~m+1\le i\le |F|$, by the analysis above, then we have $L_i(k^n)=\sum_{e\in E_i}\ell_{i,e}^n k_i^n\to 0~(n\to +\infty), ~m+1\le i\le |F|$. Since the total geodesic curvature $L_i$ is continuous on $\mathbb{R}^{|F|}_{\ge 0}$, then we have $L_i(k^n)\to L_i(k)~(n\to +\infty)$, $1\le i\le m$, i.e. $\mathcal{E}_1(k^n)\to \mathcal{E}_{1\cdots m}(k)~(n\to +\infty)$. Hence the map $\mathcal{E}$ is continuous.

For any $L\in \bar{\partial} \mathcal{L}$, there exists a small enough neighborhood $U$ of $L$ such that $\mathcal{E}^{-1}(U)$ is bounded in $\mathbb{R}^{|F|}_{\ge 0}$. We choose a sequence $\{L^n\}_{n=1}^{+\infty}\subset U$ such that $\lim_{n\to +\infty}L^n=L$. By $\{k^n\}_{n=1}^{+\infty}\subset \mathbb{R}_{\ge 0}^{|F|}$ and $k$ we denote the image of $\{L^n\}_{n=1}^{+\infty}$ and $L$ under the map $\mathcal{E}^{-1}$, respectively. Then we only need to show that $k^n\to k~(n\to +\infty)$. 

If $k^n\nrightarrow k~(n\to +\infty)$, since $\{k^n\}_{n=1}^{+\infty}\subset\mathcal{E}^{-1}(U)$ and $\mathcal{E}^{-1}(U)$ is bounded, then there exists a subsequence of $\{k^n\}_{n=1}^{+\infty}$, which is still denoted by $\{k^n\}_{n=1}^{+\infty}$, such that $k^n\to k'~(n\to +\infty)$, where $k'\ne k$. Since $\mathcal{E}$ is continuous, then $\mathcal{E}(k^n)=L^n\to \mathcal{E}(k')~(n\to +\infty)$. Besides, we know that $L^n\to L=\mathcal{E}(k)~(n\to +\infty)$, then we obtain $\mathcal{E}(k)=\mathcal{E}(k')$. This is a contradiction since the map $\mathcal{E}$ is injective. Hence the map $\mathcal{E}^{-1}$ is continuous, $\mathcal{E}$ is a homeomorphism from $\mathbb{R}^{|F|}_{\ge 0}$ to $\mathcal{L}$.    
\end{proof}

By above argument, we completed the proof of Theorem \ref{thm12}.

\section{Convergence of prescribed combinatorial Ricci flows for circle pattern metrics }\label{a7}

Given $\hat{L}=(\hat{L}_1,\cdots,\hat{L}_{|F|})^T\in\mathcal{L}_1$, then we can define the prescribed combinatorial Ricci flows as follows, i.e.
\begin{equation}\label{f1}
\frac{dk_i}{dt}=-(L_i-\hat{L}_i)k_i,~~\forall i\in F.
\end{equation}

Using change of variables $K_i=\ln k_i$, we can rewrite flows (\ref{f1}) as the following equivalent prescribed combinatorial Ricci flows
\begin{equation}\label{f2}
\frac{dK_i}{dt}=-(L_i-\hat{L}_i),~~\forall i\in F.
\end{equation}

\begin{theorem}\label{thm3}
    Given a closed topological surface $S$ with a cellular decomposition $\Sigma=(V,E,F)$, the weight $\Phi\in (0,\frac{\pi}{2})^{|E|}$ and the prescribed total geodesic curvature $\hat{L}=(\hat{L}_1,\cdots,\hat{L}_{|F|})^T\in\mathcal{L}_1$ on the face set $F$. For any initial geodesic curvature $k(0)\in\mathbb{R}_{>0}^{|F|}$, the solution of the prescribed combinatorial Ricci flows (\ref{f1}) exists for all time $t\in[0,+\infty)$ and is unique.   

\end{theorem}

\begin{proof}
We can only consider the equivalent flows (\ref{f2}). Since all $-(L_i-\hat{L}_i)$ are smooth functions on $\mathbb{R}_{>0}^{|F|}$, by Peano’s existence theorem in ODE theory, we know that the solution of flows (\ref{f2}) exists on $[0,\epsilon)$, where $\epsilon>0$. 

By Gauss-Bonnet formula, we have that 
$$\text{Area}(B_e)=2\Phi(e)-L_{i,e}-L_{j,e}>0,~~j\in F_{\{e\}},$$             
hence we obtain $0<L_{i,e}<2\Phi(e)$. Since $L_i=\sum_{e\in E_i}L_{i,e}$, we have that $0<L_i<2\sum_{e\in E_i}\Phi(e)$. Besides, since $\hat{L}\in\mathcal{L}_1$, by definition, we have $0<\hat{L}_i<\sum_{f\in F}\hat{L}_f<2\sum_{e\in E_F}\Phi(e)$. Then we obtain that $$|L_i-\hat{L}_i|\le |L_i|+|\hat{L}_i|\le 2\sum_{e\in E_i}\Phi(e)+2\sum_{e\in E_F}\Phi(e)<+\infty.$$   
Hence $|L_i-\hat{L}_i|$ is uniformly bounded by a constant, which depends only on the weight $\Phi$ and cellular decomposition $\Sigma$. By the extension theorem of solution in ODE theory, the solution of flows (\ref{f2}) exists for all time $t\in [0,+\infty)$. By existence and uniqueness theorem of solution in ODE theory, the solution of flows (\ref{f2}) is unique.       
\end{proof}

We need the following result in the theory of negative gradient flows.

\begin{lemma}(\cite{takasu}, Proposition 2.13)\label{lemma5}
Let $h: \mathbb{R}^n \rightarrow \mathbb{R}$ be a smooth convex function and let $\xi:[0, +\infty) \rightarrow \mathbb{R}^n$ be a negative gradient flow of $h$. It holds for any $\tau>0$ and $\xi^* \in \mathbb{R}^n$ that
$$
|\nabla h(\xi(\tau))|^2 \leq\left|\nabla h\left(\xi^*\right)\right|^2+\frac{1}{\tau^2}\left|\xi^*-\xi(0)\right|^2.
$$
\end{lemma}

Then we can prove the Theorem \ref{s4}.

\begin{proof}[\textbf{Proof of Theorem \ref{s4}}]
We can only consider the equivalent flows (\ref{f2}). By Theorem \ref{thm3}, we suppose that the solution of flows (\ref{f2}) is $K(t)$, $t\in [0,+\infty)$. We construct a function $\tilde{\Lambda}=\Lambda-\sum_{f\in F}K_f\hat{L}_f$, by Corollary \ref{coro1}, the function $\tilde{\Lambda}$ is convex on $\mathbb{R}^{|F|}$. Besides, we have that $\nabla\tilde{\Lambda}=\nabla\Lambda-\hat{L}=L-\hat{L},$ then we obtain  
$$\frac{dK}{dt}=-\nabla\tilde{\Lambda}(K(t)).$$ 
Since $\hat{L}\in\mathcal{L}_1$, by Theorem \ref{thm1}, there exists the unique $\hat{K}\in\mathbb{R}^{|F|}$ such that $\nabla\Lambda(\hat{K})=L(\hat{K})=\hat{L}$. By Lemma \ref{lemma5}, for any $t>0$, we have 
$$
|\nabla\tilde{\Lambda}(K(t))|^2\le |\nabla\tilde{\Lambda}(\hat{K})|^2+\frac{1}{t^2}|\hat{K}-K(0)|^2,
$$
i.e. we obtain
$$
|L(K(t))-L(\hat{K})|^2\le \frac{1}{t^2}|\hat{K}-K(0)|^2\to 0~(t\to +\infty).
$$
Hence we know that $L(K(t))\to L(\hat{K})=\hat{L}~(t\to +\infty)$, i.e. $\nabla\Lambda(K(t))\to\nabla\Lambda(\hat{K})~(t\to +\infty)$. Then by Theorem \ref{thm1}, we have that $K(t)\to \hat{K}~(t\to +\infty).$ This completes the proof. 
\end{proof}

\section{Convergence of prescribed combinatorial Ricci flows for degenerated circle pattern metrics}\label{a8}

Given the prescribed total geodesic curvature $\hat{L}\in\mathcal{L}_{i_1\cdots i_m}$, where $1\le m\le |F|-1$, $1\le i_1<\cdots<i_m\le |F|$, we can define the prescribed combinatorial Ricci flows as follows, i.e. 
\begin{equation}\label{f3}
\frac{dk_i}{dt}=-(L_i-\hat{L}_i)k_i,~~\forall i\in F_{i_1\cdots i_m},~
\frac{dk_i}{dt}=-L_ik_i,~~\forall i\in F\setminus F_{i_1\cdots i_m},
\end{equation}
where $F_{i_1\cdots i_m}=\{i_1,\cdots,i_m\}$. Then we study the existence and convergence of solution of the flows (\ref{f3}).  

Using change of variables $K_i=\ln k_i$, we can rewrite flows (\ref{f3}) as the following equivalent prescribed combinatorial Ricci flows
\begin{equation}\label{f4}
\frac{dK_i}{dt}=-(L_i-\hat{L}_i),~~\forall i\in F_{i_1\cdots i_m},~
\frac{dK_i}{dt}=-L_i,~~\forall i\in F\setminus F_{i_1\cdots i_m}.
\end{equation}
By similar argument in the proof of Theorem \ref{thm3}, we have the following theorem.

\begin{theorem}\label{thm6}
Given a closed topological surface $S$ with a cellular decomposition $\Sigma=(V,E,F)$, the weight $\Phi\in (0,\frac{\pi}{2})^{|E|}$ and the prescribed total geodesic curvature $\hat{L}\in\mathcal{L}_{i_1\cdots i_m}$ on the face set $F$. For any initial geodesic curvature $k(0)\in\mathbb{R}_{>0}^{|F|}$, the solution of the prescribed combinatorial Ricci flows (\ref{f3}) exists for all time $t\in[0,+\infty)$ and is unique.   
\end{theorem}

For simplicity, we consider the set $F_{m}=\{1,\cdots,m\}$, where $1\le m\le |F|-1$. Given $\Phi\in(0,\frac{\pi}{2})^{|E|}$, $\hat{L}=(\hat{L}_1,\cdots,\hat{L}_m,0,\cdots,0)^T\in\mathcal{L}_{1\cdots m}$, then we study the following prescribed combinatorial Ricci flows, i.e.
\begin{equation}\label{f5}
\frac{dk_i}{dt}=-(L_i-\hat{L}_i)k_i,~~1\le i\le m,~
\frac{dk_i}{dt}=-L_ik_i,~~m+1\le i\le|F|.
\end{equation}
We can also consider the equivalent prescribed combinatorial Ricci flows, i.e.
\begin{equation}\label{f6}
\frac{dK_i}{dt}=-(L_i-\hat{L}_i),~~1\le i\le m,~
\frac{dK_i}{dt}=-L_i,~~m+1\le i\le|F|.
\end{equation}

\begin{remark}\label{r1} 
For any $\tilde{K}=(\tilde{K}_1,\cdots,\tilde{K}_m)^T\in\mathbb{R}^m$, we have that $\nabla\Lambda_1(\tilde{K})=(L_1(\tilde{K}),\cdots,L_m(\tilde{K}))^T$. By (\ref{for6}), we know that $L_i(\tilde{K})(1\le i\le m)$ is actually the total geodesic curvature at the face $i$ when the radii $\mathbf{r}=(\arccot e^{\tilde{K}_1},\cdots,\arccot  e^{\tilde{K}_m},\frac{\pi}{2},\cdots,\frac{\pi}{2})^T$ or when the geodesic curvatures $k=(e^{\tilde{K}_1},\cdots,e^{\tilde{K}_m},0,\cdots,0)^T$ or when the $K=(\tilde{K}_1,\cdots,\tilde{K}_m,-\infty,\cdots,-\infty)^T$.\\
If we use geodesic curvatures $\tilde{k}=(\tilde{k}_1,\cdots,\tilde{k}_m)^T\in \mathbb{R}^{m}_{>0}$ as the variable of $L_i$, then the $L_i(\tilde{k})(1\le i\le m)$ is actually the total geodesic curvature at the face $i$ when the geodesic curvatures $k=(\tilde{k}_1,\cdots,\tilde{k}_m,0,\cdots,0)^T$. 
\end{remark}

For $\tilde{k}=(\tilde{k}_1,\cdots,\tilde{k}_m)^T\in \mathbb{R}^{m}_{>0}$, we construct the prescribed combinatorial Ricci flows as follows, i.e.
\begin{equation}\label{f8}
\frac{d\tilde{k}_i}{dt}=-(L_i(\tilde{k})-\hat{L}_i)\tilde{k}_i,~~1\le i\le m.
\end{equation}

Using change of variables $\tilde{K}_i=\ln \tilde{k}_i$, we can rewrite flows (\ref{f8}) as the following equivalent prescribed combinatorial Ricci flows
\begin{equation}\label{f7}
\frac{d\tilde{K}_i}{dt}=-(L_i(\tilde{K})-\hat{L}_i),~~1\le i\le m.
\end{equation}

By similar argument in the proof of Theorem \ref{thm3}, we have the following theorem.

\begin{theorem}\label{thm4}
Given a closed topological surface $S$ with a cellular decomposition $\Sigma=(V,E,F)$, the weight $\Phi\in (0,\frac{\pi}{2})^{|E|}$ and the prescribed total geodesic curvature $\hat{L}\in\mathcal{L}_{1\cdots m}$ on the face set $F$. For any initial geodesic curvature $\tilde{k}(0)\in\mathbb{R}_{>0}^m$ on the face set $F_m$ and 0 on the face set $F\setminus F_m$, the solution of the prescribed combinatorial Ricci flows (\ref{f8}) exists for all time $t\in[0,+\infty)$ and is unique. 
\end{theorem}

Then we study the convergence of solution to the flows (\ref{f8}).

\begin{theorem}\label{thm5}
Given a closed topological surface $S$ with a cellular decomposition $\Sigma=(V,E,F)$, the weight $\Phi\in (0,\frac{\pi}{2})^{|E|}$ and the prescribed total geodesic curvature $\hat{L}=(\hat{L}_1,\cdots,\hat{L}_{m},0\cdots,0)^T\in\mathcal{L}_{1\cdots m}$ on the face set $F$. For any initial geodesic curvature $\tilde{k}(0)\in\mathbb{R}_{>0}^m$ on the face set $F_m$ and 0 on the face set $F\setminus F_m$, the solution of the prescribed combinatorial Ricci flows (\ref{f8}) converges to the unique degenerated circle pattern metric with the total geodesic curvature $\tilde{L}=(\hat{L}_1,\cdots,\hat{L}_m)^T$ on $F_m$ and 0 on $F\setminus F_m$ up to isometry. Moreover, if the solution converges to the degenerated circle pattern metric $\hat{k}$, then $\hat{k}\in\mathbb{R}^{|F|}_{1\cdots m}$.    
\end{theorem}

\begin{proof}
We can only consider the equivalent flows (\ref{f7}). By Theorem \ref{thm4}, we suppose that the solution of flows (\ref{f7}) is $\tilde{K}(t)$, $t\in [0,+\infty)$. We construct a function $\tilde{\Lambda}_1(\tilde{K})=\Lambda_1(\tilde{K})-\sum_{f\in F_m}\tilde{K}_f\hat{L}_f$. By Corollary \ref{coro2}, the function $\tilde{\Lambda}_1$ is convex on $\mathbb{R}^{m}$. Besides, we have that 
$$\pp{\tilde{\Lambda}_1(\tilde{K})}{\tilde{K}_i}=\pp{\Lambda_1(\tilde{K})}{\tilde{K}_i}-\hat{L}_i=L_i(\tilde{K})-\hat{L}_i,~~1\le i\le m,$$
then we obtain  
$$\frac{d\tilde{K}}{dt}=-\nabla\tilde{\Lambda}_1(\tilde{K}(t)).$$

Since $\tilde{L}\in \tilde{\mathcal{L}}_{1\cdots m}$, by Theorem \ref{thm2}, there exists the unique $\tilde{K}=(\hat{K}_1,\cdots,\hat{K}_m)^T\in\mathbb{R}^m$ such that $\nabla\Lambda_1(\tilde{K})=(L_1(\tilde{K}),\cdots,L_m(\tilde{K}))^T=(\hat{L}_1,\cdots,\hat{L}_m)^T=\tilde{L}$. By Lemma \ref{lemma5}, for any $t>0$, we have 
$$
|\nabla\tilde{\Lambda}_1(\tilde{K}(t))|^2\le |\nabla\tilde{\Lambda}_1(\tilde{K})|^2+\frac{1}{t^2}|\tilde{K}-\tilde{K}(0)|^2,
$$
i.e. we obtain
$$
\sum_{i=1}^{m}(L_i(\tilde{K}(t))-\hat{L}_i)^2\le \frac{1}{t^2}|\tilde{K}-\tilde{K}(0)|^2\to 0~(t\to +\infty).
$$  
Hence for $1\le i\le m$, we know that $L_i(\tilde{K}(t))\to \hat{L}_i=L_i(\tilde{K})~(t\to +\infty)$ and $L(\tilde{K}(t))\to \tilde{L}=L(\tilde{K})~(t\to +\infty)$. Hence we have that $\nabla\Lambda_1(\tilde{K}(t))\to \nabla\Lambda_1(\tilde{K})~(t\to +\infty)$. Then by Theorem \ref{thm2}, we have that $\tilde{K}(t)\to \tilde{K}~(t\to +\infty)$. This completes the proof.     
\end{proof}

\begin{remark}\label{r2}
By Remark \ref{r1} and Theorem \ref{thm5}, we know that $(\tilde{K}_1(t),\cdots,\tilde{K}_m(t),-\infty,\cdots,-\infty)^T\\\to(\hat{K}_1,\cdots,\hat{K}_m,-\infty,\cdots,-\infty)^T~(t\to+\infty)$ and $L_i(\tilde{K}_1(t),\cdots,\tilde{K}_m(t),-\infty,\cdots,-\infty)\to \hat{L}_i~(t\to +\infty)$, $1\le i\le m$. Besides, we have that $\hat{L}_i=L_i(\hat{K}_1,\cdots,\hat{K}_m,-\infty,\cdots,-\infty)$.  
\end{remark}

For any $k\in\mathbb{R}_{i_1\cdots i_m}^{|F|}$, then $\tilde{k}=(k_{i_1},\cdots,k_{i_m})\in\mathbb{R}_{>0}^m$. Given $\hat{L}\in\mathcal{L}_{i_1\cdots i_m}$, we construct the prescribed combinatorial Ricci flows as follows, i.e.
\begin{equation}\label{f9}
\frac{d\tilde{k}_i}{dt}=-(L_i(\tilde{k})-\hat{L}_i)\tilde{k}_i,~~\forall i\in F_{i_1\cdots i_m}.
\end{equation}  
Using change of variables $\tilde{K}_i=\ln \tilde{k}_i$, we can rewrite flows (\ref{f9}) as the following equivalent prescribed combinatorial Ricci flows
\begin{equation}\label{f10}
\frac{d\tilde{K}_i}{dt}=-(L_i(\tilde{K})-\hat{L}_i),~~\forall i\in F_{i_1\cdots i_m}.
\end{equation}

We can use the same techniques above to obtain the long time existence of the solution to prescribed combinatorial Ricci flows (\ref{f9}) and Theorem \ref{s6}.

Then we study the convergence of solution to the flows (\ref{f3}).

\begin{theorem}
Given a closed topological surface $S$ with a cellular decomposition $\Sigma=(V,E,F)$, the weight $\Phi\in (0,\frac{\pi}{2})^{|E|}$ and the prescribed total geodesic curvature $\hat{L}=(\hat{L}_1,\cdots,\hat{L}_{m},0\cdots,0)^T\in\mathcal{L}_{1\cdots m}$ on the face set $F$. For any initial geodesic curvature $k(0)\in\mathbb{R}_{>0}^{|F|}$, the solution of the prescribed combinatorial Ricci flows (\ref{f5}) converges to the unique degenerated circle pattern metric with the total geodesic curvature $\hat{L}$ up to isometry. Moreover, if the solution converges to the degenerated circle pattern metric $\hat{k}$, then $\hat{k}\in\mathbb{R}^{|F|}_{1\cdots m}$. 
\end{theorem}

\begin{proof}
We consider the equivalent flows (\ref{f6}). By Theorem \ref{thm6}, we suppose that the solution of flows (\ref{f6}) is $K(t)$, $t\in [0,+\infty)$. We construct a function $\bar{\Lambda}=\Lambda-\sum_{f\in F_m}K_f\hat{L}_f$. By Corollary \ref{coro1}, the function $\bar{\Lambda}$ is convex on $\mathbb{R}^{|F|}$. Besides, we have that $\nabla\bar{\Lambda}=\nabla\Lambda-\hat{L}=L-\hat{L},$ then we obtain  
$$\frac{dK}{dt}=-\nabla\bar{\Lambda}(K(t)).$$   

By Lemma \ref{lemma5}, for any $t>0$ and $\bar{K}(t)=(\tilde{K}_1(t),\cdots,\tilde{K}_m(t),-\ln t,\cdots,\ln t)^T\in\mathbb{R}^{|F|}$, we have 
$$
|\nabla\bar{\Lambda}(K(t))|^2\le |\nabla\bar{\Lambda}(\bar{K}(t))|^2+\frac{1}{t^2}|\bar{K}(t)-K(0)|^2,
$$
By Theorem \ref{thm5}, we obtain that $\tilde{K}_i(t)\to\hat{K}_i~(t\to +\infty)$, $1\le i\le m$. Hence we know that the function $\tilde{K}_i(t)~(1\le i\le m)$ is bounded, then we have 
$$
\frac{1}{t^2}|\bar{K}(t)-K(0)|^2\to 0~(t\to +\infty).
$$
Besides, we have
$$
|\nabla\bar{\Lambda}(\bar{K}(t))|^2=\sum_{i=1}^{m}(L_i(\bar{K}(t))-\hat{L}_i)^2+\sum_{i=m+1}^{|F|}L_i^2(\bar{K}(t)).
$$
Since $\bar{K}(t)=(\tilde{K}_1(t),\cdots,\tilde{K}_m(t),-\ln t,\cdots,\ln t)^T\to(\hat{K}_1,\cdots,\hat{K}_m,-\infty,\cdots,-\infty)^T~(t\to +\infty)$, by Remark \ref{r2}, we have that $L_i(\bar{K}(t))\to L_i(\hat{K}_1,\cdots,\hat{K}_m,-\infty,\cdots,-\infty)=\hat{L}_i~(t\to +\infty)$, $1\le i\le m$. Besides, the geodesic curvature $k_i=e^{-\ln t}=\frac{1}{t}\to 0~(t\to +\infty)$, $m+1\le i\le |F|$, then $L_i(\bar{K}(t))\to 0~(t\to +\infty)$, $m+1\le i\le |F|$. Hence we have
$$|\nabla\bar{\Lambda}(\bar{K}(t))|^2\to 0~(t\to +\infty).$$
Then we obtain
$$|\nabla\bar{\Lambda}(K(t))|^2\to 0~(t\to +\infty),$$
i.e. $L(K(t))\to \hat{L}=(\hat{L}_1,\cdots,\hat{L}_{m},0\cdots,0)^T~(t\to +\infty)$.

We define some functions $k_i(t)=e^{K_i(t)}$, $1\le i\le |F|$ and some constants $\hat{k}_i=e^{\hat{K}_i}$, $1\le i\le m$. Then it is easy to know that $k(t)=(k_1(t),\cdots,k_{|F|}(t))^T$ is the solution of the prescribed combinatorial Ricci flows (\ref{f5}) and $\hat{k}=(\hat{k}_1,\cdots,\hat{k}_m,0,\cdots,0)^T\in\mathbb{R}_{1\cdots m}^{|F|}$. 

For simplicity, we can define $L(k):=L(\ln k_1,\cdots,\ln k_{|F|})$. Then we have $\hat{L}_i=L_i(\hat{k}),~1\le i\le m$ and $\hat{L}=L(\hat{k})$. Hence we know that $L(k(t))\to L(\hat{k})=\hat{L}~(t\to +\infty)$. Since $\{k(t)~|~t\in [0,+\infty)\}\subset\mathbb{R}_{>0}^{|F|}$ and $\hat{k}\in\mathbb{R}_{1\cdots m}^{|F|}$, by (\ref{for7}), we have that $\mathcal{E}(k(t))\to\mathcal{E}(\hat{k})~(t\to +\infty)$. By Theorem \ref{thm7}, $\mathcal{E}^{-1}$ is continuous, then we have that $k(t)\to\hat{k}~(t\to +\infty)$. This completes the proof. 
\end{proof}

We can use the same techniques above to obtain the Theorem \ref{s5}.

\section{Acknowledgments}
Guangming Hu is supported by NSF of China (No. 12101275). Ziping Lei is supported by NSF of China (No. 12122119). Puchun Zhou is supported by Shanghai Science and Technology Program [Project No. 22JC1400100].

\Addresses

\begin{thebibliography}{99}


\bibitem{cao} H. Cao, X. Zhu, \emph{A Complete Proof of the Poincaré and Geometrization Conjectures - application of the Hamilton-Perelman theory of the Ricci flow,} Asian J. Math. 10(2), 2006, 165-492. 


\bibitem{Tian}
X. X Chen, P. Lu, G. Tian
\emph{A note on uniformization of Riemann surfaces by Ricci,}
Proc. Amer. Math. Soc. 134 (2006), 3391-3393

\bibitem{chow} B. Chow and F. Luo, \emph{Combinatorial Ricci flows on surfaces}, J. Differential Geom. 63, 2003, 97--129.


\bibitem{GHZ} H. Ge, B. Hua, and P. Zhou,  \emph{A combinatorial curvature flow in spherical background geometry.} J.  Funct. Anal., 286(7), 2024, 110335.






\bibitem{Hamilton}R. S. Hamilton, \emph{Three-manifolds with positive Ricci curvature,} J.  Differential Geom., 17(2), 1982, 255--306.


\bibitem{nie} X. Nie, \emph{On circle patterns and spherical conical metrics}, Proc. Amer. Math. Soc., 152 (2024), 843-853.
\bibitem{Perelman1} G. Perelman, \emph{The entropy formula for the Ricci flow and its geometric applications,} arXiv: 0211159.
\bibitem{Perelman2} G. Perelman, \emph{Ricci flow with surgery on three-manifolds,} arXiv: math.DG/0303109.

\bibitem{Perelman3} G.Perelman, \emph{Finite extinction time for the solutions to the Ricci flow on certain three manifolds,} arXiv: math.DG/0307245.

\bibitem{takasu} A. Takatsu, \emph{Convergence of combinatorial ricci flows to degenerate circle patterns,} Trans.  Amer.
Math. Soc., 372(11), 2019, 7597--7617.


 
\end{thebibliography}
\end{document}